\documentclass[11pt]{amsart}
\usepackage{amssymb, latexsym, mathrsfs,   color, amsmath }
\usepackage[all]{xy}
\usepackage[colorlinks=true, pdfstartview=FitV,linkcolor=blue,citecolor=blue,urlcolor=blue]{hyperref}
\usepackage{chngcntr}
\counterwithin{figure}{section}

    \advance \topmargin by -\headheight
    \advance \topmargin by -\headsep

    \evensidemargin \oddsidemargin
\newtheorem {theorem}    {Theorem}[section]

\newtheorem {lemma}      [theorem]    {Lemma}
\newtheorem {corollary}  [theorem]    {Corollary}
\newtheorem {proposition}[theorem]    {Proposition}

\newcommand{\bb}{\mathbb}
\renewcommand{\rm}{\mathrm}

\newcommand{\cal}{\mathcal}
\newcommand{\GG}{\mathrm{G}}
\newcommand{\GGL}{\mathrm{GL}}
\newcommand{\UU}{\mathrm{U}}
\renewcommand{\sp}{\mathrm{Sp}}

\newcommand{\so}{\mathrm{SO}}
\newcommand{\Fq}{\bb{F}_q}
\newcommand{\fq}{(\bb{F}_q)}

\newcommand{\fqq}{(\bb{F}_{q^2})}
\newcommand{\Fqq}{\bb{F}_{q^2}}
\theoremstyle{definition}
\newtheorem{definition}[theorem]{Definition}

\newcommand{\CD}{{\mathcal{D}}}
\newcommand{\CQ}{{\mathcal{Q}}}

\numberwithin{equation}{section}

\newcommand{\fg}{\mathfrak{g}}
\newcommand{\co}{\mathcal{O}}
\newcommand{\wco}{\widetilde{\mathcal{O}}}

\begin{document}

\title{Wavefront sets and descent method for finite unitary groups}

\date{\today}

\author[Zhifeng Peng]{Zhifeng Peng}

\address{School of Mathematical Science, Soochow University, Suzhou 310027, Jiangsu, P.R. China}
\email{zfpeng@suda.edu.cn}

\author[Zhicheng Wang]{Zhicheng Wang$^*$}

\address{School of Mathematical Science, Soochow University, Suzhou 310027, Jiangsu, P.R. China}

\email{11735009@zju.edu.cn}

\begin{abstract}
Let $G$ be a connected reductive algebraic group defined over a finite field $\Fq$. In the 1980s, Kawanaka introduced the generalized Gelfand-Graev representations (GGGRs for short) of the finite group
$G^F$ in the case where $q$ is a power of a good prime for $G^F$. An essential feature of GGGRs is that they are very closely related to the (Kawanaka) wavefront sets of the irreducible representations $\pi$ of $G^F$. In \cite[Theorem 11.2]{L7}, Lusztig showed that if a nilpotent element $X\in G^F$ is ``large'' for an irreducible representation $\pi$, then the representation $\pi$ appears with ``small'' multiplicity in the GGGR associated to $X$. In this paper, we prove that for unitary groups, if $X$ is the wavefront of $\pi$, the multiplicity equals one, which generalizes the multiplicity one result of usual Gelfand-Graev representations. Moreover, we give an algorithm to decompose GGGRs for $\UU_n\fq$ and calculate the $\UU_4\fq$ case by this algorithm.
\end{abstract}

\maketitle

\section{Introduction}

Let $\overline{\mathbb{F}}_q$ be an algebraic closure of a finite field $\mathbb{F}_q$, which is of characteristic $p>2$. Consider a connected reductive algebraic group $G$ defined over $\Fq$ and its Lie algebra $\fg$, with Frobenius map $F$. Let $G^F$ be the group of $\Fq$-points of $G$. Let $Z$ be the center of $G$. We will assume that $q$ is large enough such that the main theorem in \cite{S} and the Jacobson-Morozov
theorem hold.
   Let $H$ be a subgroup of $G$. Let $\pi$ (resp. $\sigma$) be a representation of $G^F$ (resp. $H^F$). We write
\[
\langle \pi,\sigma \rangle_{H^F} = \dim \mathrm{Hom}_{H^F}(\pi,\sigma ).
\]
If $G=H$, we write $\langle \pi,\sigma \rangle$ instead of $\langle \pi,\sigma \rangle_{H^F}$ for short.
By abuse of notation, for a representation $\pi$, we also denote its character by $\pi$ when no confusion arises.

In the representation theory of finite groups of Lie type, Gelfand-Graev representations have become extremely useful in various contexts, especially for constructing cuspidal representations. For example, the Gelfand-Graev representation of $\GGL_n\fq$
have a multiplicity-free decomposition into generic representations. In particular, every cuspidal representation appears in the decomposition. Then we can construct modules of cuspidal representations via the Gelfand-Graev representation of $\GGL_n\fq$. However, in general, the Gelfand-Graev representation does not contain all cuspidal representations. It leads us to study a generalization of Gelfand-Graev representation.

 Generalized Gelfand-Graev representations (or GGGRs for abbreviation) of the finite group were introduced firstly by Kawanaka \cite{K1, K2, K3}, when $q$ is a power of a good prime for $G$. His original purpose is to prove Ennola's conjecture. Now, GGGRs have been extremely useful in many other contexts of representation theory; see a survey \cite{Ge}.
We may associate a nilpotent element $X \in \fg^F$ with a ${\mathfrak{sl}}$-2 triple $\gamma=\{X,H,Y\}$,
\[
[H,X]=2X,\qquad[H,Y]=-2Y,\qquad[X,Y]=H
\]
and a representation $\Gamma_\gamma$ of $G^F$
called the associated GGGR (see \cite[Section 2]{L7} or section \ref{sec5} for details). Unlike usual Gelfand-Graev representations, GGGRs no longer have multiplicity-free decomposition. In \cite[Theorem 11.2]{L7}, Lusztig showed that if a nilpotent element $X$ is ``large'' for an irreducible representation $\pi$, then the representation $\pi$ appears with ``small'' multiplicity in the GGGR $\Gamma_\gamma$ associated to $X$. The ``largest'' nilpotent element $X$
gives us an important invariant, which is the so-called ``wavefront set''. However, the multiplicity has not been worked out even in the ``largest'' case, the wavefront set case.

Howe firstly introduced the notion of the wavefront set in 1981 \cite{H2} for archimedean local fields. In the finite case, wavefront sets were first introduced by Kawanaka. Let $\pi$ be an irreducible representation of $G^F$. The \emph{Kawanaka wavefront set} $\wco$
of $\pi$ is defined to be an $F$-stable nilpotent orbit $\wco$ satisfying
 \begin{itemize}
 \item[] (WF1)  $\langle \pi,\Gamma_\gamma\rangle\ne 0$ for some $\gamma=\{X,H,Y\}$ with $X\in \wco^F$;
\item[] (WF2) $\langle \pi,\Gamma_{\gamma'}\rangle\ne 0$ for some $\gamma'=\{X',H',Y'\}$ with $X'\in \wco^{\prime F}$ implies $\rm{dim}\wco'\le\rm{dim}\wco$.
\end{itemize}
If $\gamma$ is trivial, we have $\langle \pi,\Gamma_\gamma\rangle= \rm{dim}\pi\ne 0$; hence the existence of a wavefront set of an irreducible representation is obvious. Moreover, for each irreducible representation, the uniqueness of the wavefront set was conjectured by Lusztig \cite{L2} and Kawanaka \cite{K2}, respectively. If this conjecture holds, we say the wavefront set of $\pi$ is well defined. When $p$ is sufficiently large, Lusztig \cite{L7} has shown that the wavefront sets are well defined and coincide with the ``dual'' of the unipotent
support. The existence of the unipotent support was established for $p$ good by Geck \cite{Ge1}, and for any $p$ by Geck and Malle, \cite{GM}.
Kawanaka also conjectured a description of the wavefront sets via Lusztig correspondence. This conjecture was proved by Lusztig \cite{L7} and was solved in the later work of Achar and Aubert \cite{AA} and Taylor \cite{T} with a geometric refinement of the condition (WF2).
 This result implies a deep connection between the irreducible representations of $G^F$ and the geometry of the algebraic group $G$.

In the p-adic case, the study of wavefront sets is related to the theory of automorphic forms (via their Fourier coefficients) and has found important applications in both areas. For example see \cite{Sh1,NPS,K2,Y,GRS,GRS2,JZ1,JZ2,JLZ}. The study of wavefront sets in the finite case is fascinating by itself but is also useful to the study of the p-adic case. Indeed, when the residue characteristic of a p-adic field $k$ is zero or ``sufficiently large'', DeBacker \cite{D2} parametrizes the set of $k$-rational nilpotent orbits of $G$ by equivalence classes of objects coming from the Bruhat-Tits building of the group. For p-adic fields, wavefront sets have the following variants: analytic, algebraic, geometric, unramified, and arithmetic wavefront sets (see \cite{D1, H, HC, JLZ, MW, O}). We briefly explain these notions in our previous paper \cite{PW}.

This paper is motivated by D, Jiang, D, Liu, and L Zhang's work \cite{JZ1, JZ2, JLZ} on arithmetic wavefront sets. In \cite{JZ1}, D. Jiang and L. Zhang proposed a conjecture of wavefront sets. They suggested that one can obtain the wavefront set of an irreducible representation in a generic L-packet via the local descent method in the local fields case. Recently, D. Jiang, D. Liu, and L. Zhang \cite{JLZ} define the arithmetic wavefront set of certain irreducible admissible representation $\pi$ of a classical group over the local field by the arithmetic structures of the enhanced L-parameter of $\pi$. Their definition of the arithmetic wavefront set is based on the rationality of the local Langlands correspondence and the local Gan-Gross-Prasad conjecture. They also prove that the arithmetic wavefront set is an invariant of $\pi$ (it is independent of the choice of the Whittaker datum \cite[Theorem 1.1]{JLZ}), and propose several conjectures to describe the relationship between the arithmetic wavefront sets, the analytic wavefront sets and the algebraic wavefront sets. To be more precise, the set of the arithmetic wavefront set coincides with the set of the analytic wavefront set and the algebraic wavefront set.

 In this paper, we shall prove some cases of finite fields analogy of their conjectures by known results of the finite Gan-Gross-Prasad problem \cite{LW2,LW4,Wang2}, and calculate Kawanaka wavefront sets for irreducible representations of $\UU_n\fq$. Moreover, we shall prove a multiplicity one result of GGGRs which generalizes the multiplicity one result of usual Gelfand-Graev representations. One of our main tools is the descent method of D. Ginzburg, S. Rallis and D. Soudry \cite{Gi,GRS, GRS2}, which calculates the wavefront set of automorphic cuspidal representation. In our previous, we used this method to deal with the finite symplectic groups case $\sp_{2n}\fq$. We get Kawanaka wavefront sets of some crucial irreducible representations of $\sp_{2n}\fq$, but not all irreducible representations. In this paper, we shall get the complete results for all irreducible representations of $\UU_n\fq$ by combining the descent method with a finite field analog of the theta correspondence argument in \cite{GZ, Z}.

  Generalised Gelfand-Graev representations of $\GGL_n\fq$ were explicitly described by S. Andrews and N. Thiem \cite{AT}. Rainbolt studies
the GGGs of $\UU_3\fq$ in \cite{Ra}. Thiem and  Vinroot studied the degenerate Gelfand-Graev for $\UU_n\fq$ in \cite{TV}.  In the general case, little is known about GGGRs. In this paper, we provide a new approach to calculating GGGRs.

 \subsection{main results}

Let $G$ be a connected reductive algebraic group defined over $\Fq$ with Frobenius map $F$.
Let $G^*$ be the dual group of $G$. We still denote the Frobenius endomorphism of $G^*$ by $F$.
For a semisimple element $s \in G^{*F}$, we define Lusztig series as follows:
\[
\mathcal{E}(G^F,s) = \{ \pi \in \rm{Irr}(G^F)  :  \langle \pi, R_{T^*,s}^G\rangle \ne 0\textrm{ for some }T^*\textrm{ containing }s \}
\]
where $T^*$ is an $F$-stable maximal torus of $G^*$ and $R_{T^*,s}^G$ is the Deligne-Lusztig correspondence to pair $(T^*,s)$.
And
\[
\rm{Irr}(G^F)=\coprod_{(s)}\mathcal{E}(G^F,s)
\]
where $(s)$ runs over the conjugacy classes of semisimple elements. Moreover, there is a bijection (Lusztig correspondence)
\[
\mathcal{L}_s:\mathcal{E}(G^F,s)\to \mathcal{E}(C_{G^{*F}}(s),1).
\]
The bijection $\mathcal{L}_s$ is uniquely determined in the unitary groups case. Let $a\in \overline{\mathbb{F}}_q^*$ and $[a]:=\{a^{{(-q)}^k}|k\in\bb{Z}\}$. The centralizer $C_{G^{*F}}(s)$ are of form
\[
C_{G^{*F}}(s)=\prod_{[ a] }G^{*F}_{[ a]}(s),
\]
where $G^{*F}_{[ a]}(s)$ is either a general linear group or unitary group (see subsection \ref{sec2.1}). Then there is a bijection
\[
\mathcal{L}_s:\mathcal{E}(G^F,s)\to \prod_{[a]}\mathcal{E}(G^{*F}_{[ a]}(s),1).
\]
The classification of the representations in $\mathcal{E}(G^{*F}_{[ a]}(s),1)$ was given by Lusztig and Srinivasan in \cite{LS}. So we can associate each $\pi[a]$ with a partition $\lambda[a]$, and we obtained the following classification of irreducible representations:
\[
\begin{matrix}
{\rm {Irr}}(\UU_n\fq)&\longrightarrow& \left\{\cal{L}:\cal{S}\to\cal{P}\Big|\sum_{[a]}\#[a]\cdot|\cal{L}([a])|=n\right\}\\
\\
\pi&\longrightarrow& \cal{L}_\pi
\end{matrix}
\]
where $\cal{S}$ is the set of $[a]$ and $\cal{P}$ is the set of partitions, and $\cal{L}_\pi(a)=\lambda[a]$ for each $[a]$.

For an irreducible representation $\pi$, we can associate it with a unique array ${\ell}{(\pi)}$ called descent index via the descent method in Section \ref{sec5}, and we prove the array ${\ell}{(\pi)}$ is a partition. In this paper, a partition means an array $\lambda=(\lambda_i)$ with $\lambda_i\ge \lambda_{i+1}$ for every $i$. We denote by $\lambda^t=(\lambda^t_1,\lambda^t_2,\cdots)$ the transpose of $\lambda$. Recall that the $F$-stable nilpotent orbits of a unitary group $\UU_n$ are parameterized by partitions of $n$, and there is a single $F$-rational nilpotent orbit in each $F$-stable nilpotent orbit. For such partition $\lambda$, we denote the corresponding $F$-stable nilpotent orbit by $\widetilde{\co}_{\lambda}$. We still denote the unique $F$-rational nilpotent orbit in $\widetilde{\co}_{\lambda}$ by $\co_{\lambda}$.

\begin{theorem}\label{main1}
 Let $\pi\in\rm{Irr}(\UU_n\fq)$, $\ell(\pi)=(\ell_1,\ell_2,\cdots)$ be the descent index of $\pi$ (see section \ref{sec5} for detail). Then we have
 \[
 \ell_i=\sum_{[a]}\#[a]\cdot\cal{L}_\pi([a])^t_i
 \]
 and
  \[
 \rm{dim}\rm{Wh}_{\co_{\ell(\pi)}}(\pi)=1,
 \]
 where $\rm{Wh}_{\co_{\ell(\pi)}}(\pi)$ is the generalized Whittaker model of $\pi$ with respect to $F$-rational nilpotent orbit $\co_{\ell(\pi)}$. Moreover, the $F$-stable nilpotent orbit by $\widetilde{\co}_{\ell(\pi)}$ is the wavefront set of $\pi$.

\end{theorem}

In our previous work, we calculated the wavefront sets for some irreducible representations of finite symplectic groups. However, we can not get the wavefront set of each irreducible representation by the descent method. The descent indexes might not even be a partition. The reason is that  we only have the Fourier-Jacobi descent in the symplectic groups case, which gives us descent indexes only consisting of even integers. However, many irreducible representations with wavefront sets correspond to a partition consisting of some odd integers. For example, the wavefront set of the trivial representation of symplectic groups corresponds to the partition $(1,1,\cdots,1)$. But the descent index of the trivial representation is $(0,2,0,2\cdots,0,2)$. Fortunately, the descent index is not far from the wavefront set. We can get the partition of the wavefront set for some irreducible representations by ``taking the average of two adjacent numbers of the descent index'' \cite{PW}. In the unitary groups case, we have both the Fourier-Jacobi and Bessel descent. So the descent indexes would consist of both odd and even integers. Roughly speaking, in this paper, we would consider one of the Fourier-Jacobi descent and the Bessel descent, which lives in a smaller group, and call it for the first descent. The first descent is always an irreducible representation of a unitary group which is different from the symplectic groups case. The irreducibility is crucial in our proof and implies to the multiplicity one result in Theorem \ref{main1}. Actually, consider the trivial representation symplectic group, the Fourier-Jacobi descent is not irreducible. That's why ``$(0,2)$'' appears in the descent index instead of $(1,1)$. However, in the unitary case, the first descent of the trivial character of $\UU_n\fq$ is the trivial character of $\UU_{n-1}\fq$ and the descent index is precisely the partition $(1,1\cdots,1)$.

Let us outline the strategy of the proof. For a nilpotent orbit $\co_\lambda=\co_{(\lambda_1,\cdots,\lambda_k)}$ of $\UU_n\fq$, consider the following nilpotent orbits:
\[
\begin{array}{ccccc}
\co_1:=&\co_{(\lambda_1,1,1,\cdots,1)}&\subset& \UU_{n_1}:=&\UU_n\fq\\
\co_2:=&\co_{(\lambda_2,1,1,\cdots,1)}&\subset &\UU_{n_2}:=&\UU_{n-\lambda_1}\fq\\
\co_3:=&\co_{(\lambda_3,1,1,\cdots,1)}&\subset& \UU_{n_3}:=&\UU_{n-\lambda_1-\lambda_2}\fq\\
\vdots&\vdots&&\vdots\\
\co_k:=&\co_{(\lambda_k)}&\subset& \UU_{n_k}:=&\UU_{\lambda_k}\fq\\
\end{array}.
\]
Since the stabilizer of the orbits $\co_i$ in $\UU_{n_i}\fq$ is the next unitary group $\UU_{n_{i+1}}\fq$, the generalized Whittaker model
$
 \rm{Wh}_{\co_1}( \pi  )
$
is a representation of $\UU_2\fq$. Then we can define the generalized Whittaker model of generalized Whittaker model
\[
  \rm{Wh}_{\co_2}( \rm{Wh}_{\co_1}( \pi  ))
\]
and inductively
\[
  \rm{Wh}_{\co_i}( \cdots \rm{Wh}_{\co_2}( \rm{Wh}_{\co_1}( \pi  ))).
\]
We shall apply a finite field analog of the root exchanging technique of Ginzburg-Rallis-Soudry \cite{GRS, GRS2} to prove the following equation (Section \ref{sec6}):
\[
 \rm{dim}\rm{Wh}_{\co_\lambda}(\pi)= \rm{dim} \rm{Wh}_{\co_k}( \cdots \rm{Wh}_{\co_2}( \rm{Wh}_{\co_1}( \pi  )  )  ).
\]
Therefore, we only have to calculate this kind of generalized Whittaker model:
\[
 \rm{Wh}_{\co_{\ell,1,1,\cdots,1}}( \pi  ).
\]
Then the calculation of this kind of generalized Whittaker model has been completed in previous works \cite{LW3, Wang2}. In particular, if $\ell$ is the first occurrence index of $\pi$ (see Section \ref{sec5}), then the generalized Whittaker model $ \rm{Wh}_{\co_{\ell,1,1,\cdots,1}}( \pi  )$ is irreducible, which brings us to the multiplicity one result. In general, for an arbitrary nilpotent orbit $\co_\lambda=\co_{\lambda_1,\cdots,\lambda_k}$, the root exchanging technique still works even if $\lambda_i$ is not the first occurrence index. So, in principle, one can calculate the dimension of the generalized Whittaker model for any nilpotent orbits. We shall calculate the $\UU_4\fq$ case in Section \ref{sec8}.
However, in the general case, although the generalized Whittaker model $  \rm{Wh}_{\co_i}( \cdots \rm{Wh}_{\co_2}( \rm{Wh}_{\co_1}( \pi  )))$ still can be calculated explicitly, it contains a large number of irreducible constitutes, and the description of the composition of generalized Whittaker models would be very complicated.

In \cite[Section 1.5]{T}, there is a
geometric refinement of the condition (WF2) in the definition of wavefront sets as follows:
 \begin{itemize}
\item[] (WF$2'$) $\langle \pi,\Gamma_{\gamma'}\rangle\ne 0$ for some $\gamma'=\{X',H',Y'\}$ with $X'\in \wco_{\lambda'}^{F}$ implies $\widetilde{\co}_{\lambda'}\subset\widetilde{\co}_\lambda$.
\end{itemize}
Although we can not calculate the dimension the generalized Whittaker model for any nilpotent orbits, we still get a non-vanishing result, which implies the opposite direction of (WF$2'$).
\begin{theorem}\label{main2}
 Let $\pi\in\rm{Irr}(\UU_n\fq)$, $\ell(\pi)=(\ell_1,\ell_2,\cdots)$ be the descent index of $\pi$. If $\lambda\le \ell(\pi) $, then
  \[
 \rm{dim}\rm{Wh}_{\co_{\lambda}}(\pi)>0.
 \]
 where  $\lambda\le \ell(\pi) $ is the parabolic ordering
 \[
\sum_{i=1}^k \lambda_i\le \sum_{i=1}^k\ell(\pi)_i.
 \]

\end{theorem}

This paper is organized as follows. In Section \ref{sec2}, we recall the theory of Deligne-Lusztig characters and the Lusztig correspondence and give a classification of irreducible representations of unitary groups. In Section \ref{sec3}, we review the parametrization of nilpotent orbits in the classical Lie algebras, following the book by Collingwood and McGovern \cite{CM} and the construction of the generalized Gelfand-Graev representations, followed by Luszig \cite{L7} and Gomez-Zhu \cite{GZ}. In section \ref{sec4}, we recall the theta correspondence over finite fields for dual pairs of two unitary groups. Then we recall the behavior of nilpotent orbit with theta correspondence in \cite{GZ} and write a finite fields version of their results. In Section \ref{sec5}, we describe the relationship between the generalised Gelfand-Graev representation and the descent method, and calculate the descent of irreducible representations. In Section \ref{sec7}, we obtain a finite unitary groups version of ``Exchanging roots Lemma'', which is the main tools to prove our theorems. We point out that the proof of ``Exchanging roots Lemma'' differs from the proof in \cite{GRS, PW} and is based on the theta argument on nilpotent orbit described in section \ref{sec4}. In Section \ref{sec7}, we prove our main theorems. In Section \ref{sec8}, we calculate generalised Gelfand-Graev representations of $\UU_4\fq$.
\subsection*{Acknowledgement}  We acknowledge generous support provided by National Natural Science Foundation of PR China (No. 12071326 and No. 12201444), China Postdoctoral Science Foundation (No. 2021TQ0233, No. 2021M702399), and Jiangsu Funding Program for Excellent Postdoctoral Talent (No. 2022ZB283029).

\section{Deligne-Lusztig characters and Lusztig correspondence} \label{sec2}
We review some standard facts on Deligne-Lusztig characters and Lusztig correspondence (cf. \cite[Chapter 7, 12]{C}). Let $G$ be a connected reductive algebraic
group over $\mathbb{F}_q$. In \cite{DL}, Deligne and Lusztig defined a virtual character $R^{G}_{T,\theta}$ of $G^F$, associated to an $F$-stable maximal torus $T$ of $G$ and a character $\theta$ of $T^F$.
More generally, let $L$ be an $F$-stable Levi subgroup of a parabolic subgroup $P$ which is not necessarily $F$-stable, and $\pi$ be a representation of the group $L^F$. Then $R^G_L(\pi)$ is a virtual character of $G^F$. If $P$ is $F$-stable, then the Deligne-Lusztig induction coincides with the parabolic induction
\[
R^G_L(\pi)= \rm{Ind}^{G^F}_{P^F}(\pi).
\]

\subsection{Centralizer of a semisimple element }\label{sec2.1}

 Let $G$ be a unitary group over finite field, and $F$ be a Frobenius morphism of $G$, and $s$ be a semisimple element in $G$. Let $C_{G(\overline{\bb{F}}_q)}(s)$ be the centralizer in $G(\overline{\bb{F}}_q)$ of a semisimple element $s $. In \cite[subsection 1.B]{AMR}, $C_{G(\overline{\bb{F}}_q)}(s)$ is described as follows. We denote $T(\overline{\bb{F}}_q) \cong \overline{\bb{F}}^\times_q\times\cdots\times\overline{\bb{F}}_q^\times$ by a $F$-rational maximal torus of $G(\overline{\bb{F}}_q)$, and by $s=(x_1,\cdots,x_l)\in T^F$. Let $\nu_{a}(s):=\#\{i|x_i=a\}$ and
  \[
  [ a] := \{a^{{(-q)}^k}|k\in\bb{Z}\}.
  \]
  Clearly, if $a'\in[a]$ and $a\in\{x_i\}$, then $a'\in\{x_i\}$ and $\nu_{a'}(s)=\nu_{a}(s)$.
 The group $C_{G(\overline{\bb{F}}_q)}(s)$ has a natural decomposition with the eigenvalues of $s$:
\[
C_{G(\overline{\bb{F}}_q)}(s)=\prod_{[ a] \subset \{x_i\}}G_{[ a]}(s)(\overline{\bb{F}}_q)
\]
where $G_{[ a]}(s)(\overline{\bb{F}}_q)$ is either a general linear group $\GGL_{\nu_a(s)}$ or unitary group $\UU_{\nu_a(s)}$.

For $[a]\nsubseteq \{x_i\}$, we set $G_{[ a]}(s):=1$. Then we rewrite
\begin{equation}\label{decomp}
C_{G^F}(s)=\prod_{[ a] }G^F_{[ a]}(s)
\end{equation}
where the product runs over every $[a]$ with $a\in \overline{\bb{F}}_q$.

\subsection{Lusztig correspondence for unitary groups}\label{sec2.2}
Let $G^*$ be the dual group of $G$. We still denote the Frobenius endomorphism of $G^*$ by $F$. Then there is a natural bijection between the set of $G^F$-conjugacy classes of $(T, \theta)$ and the set of $G^{*F}$-conjugacy classes of $(T^*, s)$ where $T^*$ is a $F$-stable maximal torus in $G^*$ and $s \in   T^{*F}$. We will also denote $R_{T,\theta}^G$  by $R_{T^*,s}^G$ if $(T, \theta)$ corresponds to $(T^*, s)$.
For a semisimple element $s \in G^{*F}$, define
\[
\mathcal{E}(G^F,s) = \{ \chi \in \rm{Irr}(G^F)  :  \langle \chi, R_{T^*,s}^G\rangle \ne 0\textrm{ for some }T^*\textrm{ containing }s \}.
\]
The set $\mathcal{E}(G^F,s)$ is called the Lusztig series. We can thus define a partition of $\rm{Irr}(G^F)$ by Lusztig series
i.e.,
\[
\rm{Irr}(G^F)=\coprod_{(s)}\mathcal{E}(G^F,s).
\]
An irreducible representation $\pi$ of $G^F$  is called unipotent if
\[
\pi\in \cal{E}(G^F,I).
\]
\begin{proposition}[Lusztig]\label{Lus}
There is a bijection
\[
\mathcal{L}_s:\mathcal{E}(G^F,s)\to \mathcal{E}(C_{G^{*F}}(s),I),
\]
extended by linearity to a map between virtual characters satisfying that
\begin{equation}\label{Lus2}
\mathcal{L}_s(\varepsilon_G R^G_{T^*,s})=\varepsilon_{C_{G^{*}}(s)} R^{C_{G^{*F}}(s)}_{T^*,1}.
\end{equation}
Moreover, we have
\[
\rm{dim}(\pi)=\frac{|G^F|_{p'}}{|C_{G^{*F}}(s)|_{p'}}\rm{dim}(\cal{L}_s(\pi))
\]
where $|G|_{p'}$ denotes greatest factor of $|G|$ not divided by $p$ with $\Fq=\bb{F}_{p^m}$, , and  $\varepsilon_G:= (-1)^r$ where $r$ is the $\Fq$-rank of $G$.
In particular, Lusztig correspondence sends cuspidal representation to cuspidal representation.
\end{proposition}

The correspondence $\cal{L}_s$ is usually not uniquely determined. In this paper, we only
 consider Lusztig correspondence for unitary groups, which is uniquely determined by (\ref{Lus2}). By (\ref{decomp}), we can rewrite Lusztig correspondence as
\begin{equation}\label{lc}
\begin{matrix}
\mathcal{L}_s:&\mathcal{E}(G^F,s)&\to &\prod_{[ a] }\mathcal{E}(G^{*F}_{[ a]}(s),I)&\\
\\
&\pi&\to&\prod_{[a]}\pi[a]
\end{matrix}
\end{equation}
where $\pi[a]$ is an irreducible unipotent representation of a general linear or unitary group.

\subsection{Classification of irreducible representations of unitary groups}\label{sec2.3}

The classification of the representations of $G^F=\GGL_n\fq$ and $\UU_n(\Fq)$ was given by Lusztig and Srinivasan in \cite{LS}. Denote by $W_n\cong S_n$ the Weyl group of the diagonal torus $T_0$ in  $\GGL_n\fq$ or $\UU_n(\Fq)$. For any $F$-stable maximal torus $T$, there is $g\in G$ such that $^gT=T_0$. Since $T$ is $F$-stable, we have $gF(g^{-1})\in N_G(T_0)$. If $w$ is the image of  $gF(g^{-1})$ in $W_n$, then we denote $T$ by $T_w$.

\begin{theorem}\label{thm3.3}
Let $\sigma$ be an irreducible representation of $S_n$. Then
\[
R_\sigma^{\UU_n}:=\frac{1}{|W_n|}\sum_{w\in W_n}\sigma(ww_0)R_{T_w,1}^{\UU_n}
\]
is (up to sign) a unipotent representation of $\UU_n(\Fq)$ (resp. $\GGL_n\fq$) and all unipotent representations of $\UU_n(\Fq)$ (resp. $\GGL_n\fq$) arise in this way.\end{theorem}

It is well-known that irreducible representations of $S_n$ are parametrized by partitions of $n$. For a partition $\lambda$ of $n$,  denote by  $\sigma_\lambda$ the corresponding representation of $S_n$, and
let $\pi_\lambda= \pm R_{\sigma_\lambda}^{\UU_n}$ (resp. $R_{\sigma_\lambda}^{\GGL_n}$) be the corresponding unipotent representation of $\UU_n(\Fq)$ (resp. $\GGL_n\fq$).
By Lusztig's result \cite{L1},  $\pi_\lambda$ is (up to sign) a unipotent cuspidal representation of $\UU_n(\Fq)$ if and only if $n=\frac{k(k+1)}{2}$ for some positive integer $k$ and  $\lambda=[k,k-1,\cdots,1]$.

Let $\lambda=(\lambda_1,\lambda_2,\cdots,\lambda_k)$ be a partition of $n$, and $|\lambda|:=\sum_i\lambda_i$. For simplicity of notations, we write $\lambda=(a_1^{k_1},\cdots,a_l^{k_l})$ where $a_j$ runs over $\{\lambda_i\}$ and $k_j=\#\{i|\lambda_i=a_j\}$.
We set
\[
a\lambda:=(a\lambda_1,a\lambda_2,\cdots,a\lambda_k) \ \textrm{ with }a\in\bb{Z}
\]
and
\[
\lambda^{-j}:=(\lambda_1-j,\lambda_2-j,\cdots,\lambda_k-j).
\]
It is obvious that one can get $\lambda^{-j}$ by removing the first $j$ columns of $\lambda$. As is standard, we realize partitions as Young diagrams, and denote by $\lambda^t=(\lambda^t_1,\lambda^t_2,\cdots)$ the transpose of $\lambda$. Then $(\lambda^{-j})^t_1=\lambda^t_{j+1}$.

Let $G=\UU_n$ be a unitary group, and $G^*$ be the dual group. Recall that we have Lusztig correspondence.
\[
\begin{matrix}
\mathcal{L}_s:&\mathcal{E}(G^F,s)&\to &\prod_{[ a] }\mathcal{E}(G^{*F}_{[ a]}(s),I)&\\
\\
&\pi&\to&\prod_{[a]}\pi[a]
\end{matrix}
\]
where $G^{*F}_{[ a]}(s)$ is either a unitary group or a general linear group, and $\pi[a]$ are irreducible unipotent representations of $G^{*F}_{[ a]}(s)$. Then we can associate each $\pi[a]$ with a partition $\lambda[a]$, and we obtained the following classification of irreducible representations:
\[
\begin{matrix}
{\rm {Irr}}(\UU_n\fq)&\longleftrightarrow& \left\{\cal{L}:\cal{S}\to\cal{P}\Big|\sum_{[a]}\#[a]\cdot|\cal{L}([a])|=n\right\}\\
\\
\pi&\longleftrightarrow& \cal{L}_\pi
\end{matrix}
\]
where $\cal{S}$ is the set of $[a]$ and $\cal{P}$ is the set of partitions, and $\cal{L}_\pi(a)=\lambda[a]$ for each $[a]$. For example, let $\bf{1}$ be the trivial representation of $\UU_n\fq$. Then
 \[
\cal{L}_{\bf{1}}[a]=\left\{
  \begin{aligned}
 &{(n)}&\textrm{ if $a=1$};\\
  &{(0)}&\textrm{ otherwise}.
\end{aligned}
\right.
\]

\section{Nilpotent Orbits and generalized Gelfand-Graev representations}\label{sec3}
Let $G$ be a connected reductive algebraic group defined over $\Fq$ and $\fg$ its Lie algebra, on which we
fix an Ad $G$-invariant non-degenerate bilinear form $\kappa$.

\subsection{$\frak{sl}_2$-triples}\label{sec3.1}
The standard references for the classification of nilpotent orbits can be found in
\cite{CM} or \cite[Chapter 5]{C}.
We will review the basic results on nilpotent orbits and $\frak{sl}_2$-triples in $\fg$.

A $\frak{sl}_2$-triple is a Lie algebra homomorphism $\gamma:\frak{sl}_2\to \fg$.
Let $\Theta$ be the set of all $\frak{sl}_2$-triples and for $\gamma\in \Theta$, we set
\[
X:=\gamma\begin{pmatrix}0&1\\
0&0
\end{pmatrix},\
Y:=\gamma\begin{pmatrix}0& 0\\
1& 0
\end{pmatrix}
,\
H:=\gamma\begin{pmatrix}1& 0\\
0& -1
\end{pmatrix}.
\]
We regard $\Theta$ as a closed subvariety of the vector space $\fg^{\otimes 3}$, via $\gamma\to \{X,Y,H\}$.

For a fixed $\frak{sl}_2$-triple,
let $\fg_i = \{Z \in \fg | \rm{ad}(H)(Z) = iZ\}$ with $i \in \bb{Z}$. Then, from standard $\frak{sl}_2$-theory, we
have the decomposition
\[
\fg = \bigoplus_{j\in\bb{Z}}\fg_{j}.
\]
Let
\[
\fg_{\ge i}:=\bigoplus_{j\ge i}\fg_{j}\textrm{ and } \fg_{\le i}:=\bigoplus_{j\le i}\fg_{j}.
\]
They are Lie algebras of a close connect subgroup $G_{\ge i}$ and $G_{\le i}$ of $G$, respectively.

From the well-known results of Jacobson-Morozov and Kostant \cite[Chapter
3]{CM}, there is a 1-1 correspondence
\[
\begin{matrix}
\left\{\begin{matrix}
\textrm{Ad}(G)\textrm{ conjugacy classes of}\\
 \textrm{ $\frak{sl}_2$-triples in $\fg$}
\end{matrix}
\right\}
&\longleftrightarrow&
\left\{\begin{matrix}
\textrm{Nonzero nilpotent Ad}(G)\textrm{-orbits}\\
\wco\subset\fg
\end{matrix}
\right\}\\
\\
\gamma = \{H, X, Y \}&\longleftrightarrow&\wco=\rm{Ad}(G)\cdot X
\end{matrix}.
\]
If the conjugacy class of an $\frak{sl}_2$-triple $\gamma$ corresponds to a nilpotent orbit $\wco\subset\fg$, then
we say that $\gamma$ is an $\frak{sl}_2$-triple of type $\wco$. We also call these $\rm{Ad}(G)$-orbits $\wco$ the $F$-stable nilpotent orbits in $\fg$.
\subsection{$F$-rational orbits of unitary groups}\label{sec3.2}
This subsection will classify $F$-rational orbits $\co$ for each $F$-stable nilpotent orbit $\wco$.

From now on, assume that $G$ is a unitary group, and $G^F=\UU_{n}\fq$ naturally acts on $V$ where $V$ is an $\Fq$-vector space endowed with a Hermitian form.
Recall that the set of $F$-stable nilpotent orbits of unitary groups are parameterized by partitions of $n$. Hence each partition $\lambda$
defines a $F$-stable nilpotent orbit $\wco_\lambda$.

Methods of \cite[I.6]{Wal} imply the parameterization of $F$-rational nilpotent orbit. It also can be found in \cite[section 3]{GZ}.
The $F$-rational orbits of nilpotent elements in $\fg^F$ are parameterized by pairs $(\lambda,(Q(k_i)))$, where $\lambda=(a_1^{k_1},\cdots,a_m^{k_m})$ and $Q(k_i)$ is a equivalence class of non-degenerate Hermitian form of dimension $k_i$ with $1\le i\le m$. Recall that in the finite fields case, for each $k\in \bb{N}$, there is only one equivalence class of the non-degenerate Hermitian form of dimension $k$. Then we have the following parameterization of $F$-rational orbits:
\begin{theorem}\label{ra}
Each $F$-stable nilpotent orbit $\wco$ contains a singleton $F$-rational orbit $\co$. In particular, the $F$-rational nilpotent orbits of $G$ are parameterized by partitions $\lambda$ of $n$.
\end{theorem}

\subsection{Generalized
Gelfand-Graev representations}\label{sec3.3}
We fix an $\rm{Ad}(G)$-invariant non-degenerate bilinear form $\kappa$ on $\fg$. Let $\psi$ be a non-trivial character of $\Fq$.

Pick a $\frak{sl}_2$-triple $\gamma=\{X,Y,H\}\in \Theta^F$. Then $G_{\ge i}$ and $\fg_{\ge i}$ are defined over $\Fq$.
 We regard $\kappa$ as a linear form of $\fg_{\ge 1}$ defined by
\[
\kappa(x)  =  \kappa(Y, x).
\]
By setting $\langle x,  y\rangle  =  \kappa ([x, y])$, we define a non-singular symplectic form on $\fg_1$.
We choose a Lagrangian subspace $L$, for $\langle  , \rangle$ and we denote by$\fg_{\ge1.5}:=L\oplus\bigoplus_{\ge2}\fg_i$.  Then $\fg_{\ge1.5}$ is the Lie algebra of a closed, connected unipotent subgroup $G_{\ge1.5}$ of $G_{\ge1}$. The restriction of $\kappa$ to $\fg_{\ge1.5}$ vanishes on all commutators of this Lie algebra; hence, the composition
\[
G_{\ge1.5}^F\xrightarrow{\rm{log}}\fg_{\ge1.5}^F\xrightarrow{\kappa}\Fq\xrightarrow{\psi}\bb{C}^\times
\]
is a homomorphism of algebraic groups. By abuse of notations, we still denote the above character of $G_{\ge1.5}^F$ by $\psi$.

We now associate $\gamma$ with a generalized Gelfand-Graev representation.
\[
\Gamma_\lambda:=\rm{Ind}^{G^F}_{G_{\ge 1.5}^F}\psi
\]
and
\[
q^{\rm{dim}\frac{\fg_1}{2}}\cdot\Gamma_\lambda=\rm{Ind}^{G^F}_{G_{\ge 2}^F}\psi.
\]

Note that $H_\lambda:=G_{\ge 1}^F/[G_{\ge 2}^F,G_{\ge 2}^F]\cong L^F\oplus L^{\vee F} \oplus Z^F$ is a Heisenberg group where $Z = G_{\ge 2}^F/[G_{\ge 2}^F,G_{\ge 2}^F]$ and $L$ Lagrangian subspace of $\fg_1$ as before.
 Then, according to the Stone-von Neumann
theorem, there exists a unique, up to equivalence, irreducible representation $\rho_{\psi}$ of $H_\lambda$ such that the center $Z^F$ of $H_\lambda$ acts by the character $\psi$. By composition with the natural homomorphism from $G_{\ge 1}^F$ to $H_\lambda$, we regard $\rho_{\psi}$ as a representation of $G_{\ge 1}^F$, and according to the uniqueness of $\rho_{\psi}$, we have
\[
\rho_{\psi}=\rm{Ind}^{G_{\ge 1}^F}_{G_{\ge 1.5}^F}\psi.
\]
Since
\[
M^F_X = G^F_\lambda := \{g \in G^F |\rm{Ad}(g)X = X, \rm{Ad}(g)H = H, \rm{Ad}(g)Y = Y \}.
\]
 preserves $\gamma$, it is well-known that there exists a representation of a semi-direct product
$M_X^F\rtimes G_{\ge 1}^F$ which extends the representation $\rho_{\psi }$ of $G_{\ge 1}^F$. We refer to the representation $\omega_{\psi }$ of $M^F_X\rtimes G^F_{\ge 1}$ as the Weil representation associated to $\psi$. Then for a representation $\pi$ of $G^F$, the Hom space
\[
\rm{Hom}_{G_{\ge 1}^F}(\pi,\omega_{\psi })
\]
is naturally a representation of $M_X^F$. In the unitary groups case, the subgroup $M_X$ is also a unitary group. Recall that different choices of character $\psi$ will give the same Weil representation $\omega_{\psi }$ of the unitary group. So the above Hom space does not depend on the choice of character $\psi$.
By Frobenius reciprocity,
\[
\rm{Hom}_{G_{\ge 1}^F}(\pi,\omega_{\psi })\cong\rm{Hom}_{G^F}(\pi,\rm{Ind}^{G^F}_{G_{\ge 1.5}^F}\psi)\cong \rm{Hom}_{G_{\ge 1.5}^F}(\pi,\psi).
\]
If the $F$-rational orbit of $X$ is corresponding to $\lambda$, then we denote this representation of $M_X^F$ by $\rm{Wh}_{\co_\lambda}(\pi)$. We say $\pi$ has a nontrivial generalized Whittaker model on $\co_\lambda$ or $\co_\lambda$ supports $\pi$ if  $\rm{Wh}_{\co_\lambda}(\pi)\ne 0$.

If $\lambda=(\ell,1^{2(n-\ell)})$, then $M_X\cong\UU_{n-\ell}$ and $\rm{Wh}_{\co_\lambda}(\pi)$ is a representation of $\UU_{n-\ell}\fq$. Let $G'=\UU_{n-\ell}$ and $\co_{\lambda'}$ be an $F$-rational nilpotent orbit of $G'$. We say $\co_\lambda\circ\co_{\lambda'}$ supports $\pi$ if
\[
\rm{Wh}_{\co_{\lambda'}}(\rm{Wh}_{\co_\lambda}(\pi))\ne 0.
\]

For a reductive subgroup $R^F$ of $M^F_X$ and a representation $\tau$ of $R^F$, let
\[
\rm{Wh}_{\co,\tau} (\pi) := \textrm{Home}_{R^F\ltimes N^F} (\pi , \tau \otimes \omega_\psi).
\]
(When $R$ is the trivial group, the above model becomes a generalized Whittaker model.)

\section{Theta correspondence and moment map on unitary groups}\label{sec4}

In this section, we recall the image of nilpotent orbits under theta correspondence in \cite{GZ, Z} and write a finite fields version of their results for unitary groups. We follow the notations of \cite{Z}.

\subsection{Theta correspondence for finite unitary groups}
Let $\bb{F}=\Fq$ be a finite field, and $\bb{E}$  be the quadratic field extension of $\bb{F}$. Let $\epsilon\in \{\pm 1\}$, and $V$ be an $\epsilon$-Hermitian $\bb{E}$-vectors space. Let $V'$ be a $-\epsilon$-Hermitian $\bb{E}$-vectors space. Let $\UU(V)$ and $\UU(V')$ be the isometry group of $V$ and $V'$, respectively. Let $W =: \textrm{Hom}_{\bb{E}}(V, V ')$, and we define a symplectic form on $W$ by
setting
\[
\langle T, S\rangle_W:= \textrm{Tr}_{\bb{E}/\bb{F}}(T^*S)
\]
where $\textrm{Tr}_{\bb{E}/\bb{F}}$ is the trace of $T^*S$ as an $\bb{F}$-linear transformation, and $T^*$ is defined by
\[
(T v, v')_{V '} = (v, T^*v')_V ,\ v \in V ,\ v' \in V'.
\]
Let $\sp(W)$ be the isometry group of $\langle \cdot, \cdot\rangle_W$. There is a natural homomorphism: $\UU(V) \times \UU(V') \to \sp(W)$ given by
\[
(g, g') \cdot T = g'T g^{-1}\
for\ T \in \textrm{Hom}_{\bb{E}}(V, V'),\  g \in \UU(V),\  g' \in \UU(V').
\]
We call $(\UU(V) ,\UU(V'))$ a dual pair of $\sp(W)$. To ease notation, we will skip $\bb{E}$ in $\textrm{Hom}_{\bb{E}}(V, V')$ from now on.

Assume that $\rm{dim}W=2N$. Then $\sp(W)\cong \sp_{2N}\fq$.
Let $\omega_{\rm{Sp}_{2N}}$ be the Weil representation (cf. \cite{Ger}) of the finite symplectic group $\rm{Sp}_{2N}(\Fq)$, which depends on a  nontrivial additive character $\psi$ of $\Fq$.
Write
$\omega_{\UU(V),\UU(V')}$ for the restriction of $\omega_{\rm{Sp}_{2N}}$ to $\UU(V) \times\UU(V')$. Then it decomposes into a direct sum
\[
\omega_{\UU(V),\UU(V')}=\bigoplus_{\pi,\pi'} m_{\pi,\pi '}\pi\otimes\pi '
\]
where $\pi$ and $\pi '$ run over irreducible representations of $\UU(V)$ and $\UU(V')$ respectively, and $m_{\pi,\pi'}$ are nonnegative integers.. We can rearrange this decomposition as
\[
\omega_{\UU(V),\UU(V')}=\bigoplus_{\pi} \pi\otimes\Theta_{\UU(V),\UU(V')}(\pi )
\]
 where $\Theta_{\UU(V), \UU(V')}(\pi ) = \bigoplus_{\pi'} m_{\pi,\pi '}\pi '$ is a (not necessarily irreducible) representation of $\UU(V')$, called the (big) theta lifting of $\pi$ from $\UU(V)$ to $\UU(V')$. Write $\pi'\subset \Theta_{\UU(V),\UU(V')}(\pi)$ if $\pi\otimes\pi'$ occurs in $\omega_{\UU(V),\UU(V')}$, i.e. $m_{\pi, \pi'}\neq 0$.  We remark that even if $\Theta_{\UU(V),\UU(V')}(\pi)=\pi'$ is irreducible, one only has
 \[
 \pi\subset \Theta_{\UU(V'),\UU(V)}(\pi'),
 \]
 while equality does not necessarily hold.

If $(\UU(V),\UU(V'))\cong(\UU_n\fq, \UU_{n'}\fq)$, then we denote $\omega_{\UU(V), \UU(V')}$ by $\omega_{n, n'}$, and $\Theta_{\UU(V),\UU(V')}$ by $\Theta_{n,n'}$. In particular, we denote by $\omega_n$ the restriction of $\omega_{\rm{Sp}_{2n}}$ to $\rm{U}_n(\Fq)$.

By  \cite[Lemma 6.2 and Proposition 6.4]{LW2}, we have the following compatibility for the theta lifting and parabolic induction.

\begin{proposition} \label{theta}
Let $\GG_\ell:=\rm{Res}_{\mathbb{F}_{q^2}/\Fq}\GGL_\ell$. Let $\tau$ be an irreducible cuspidal representation of $\GG_\ell\fq$ which is not conjugate self-dual, $\pi$ be an irreducible unipotent representation of $\UU_n(\Fq)$, and  $\pi':=\Theta_{n,n'}(\pi)$. Then we have
\[
\Theta_{n+2\ell, n'+2\ell} (R^{\UU_{n+2\ell}}_{\GG_\ell\times \UU_{n}}(\tau \otimes \pi))= R^{\UU_{n'+2\ell}}_{\GG_\ell\times \UU_{n'}}(\tau\otimes \pi').
\]
\end{proposition}

Theta correspondence over finite fields is described explicitly in \cite{AMR, P1, P2}. The following theorem (see \cite[Theorem 2.6]{AMR} or \cite[Proposition 5.15]{P2}) describes the relationship between Theta correspondence and Lusztig correspondence.
\begin{theorem}\label{p1}
 Let $(G, G') = (\UU_{n},\UU_{n'})$, and $\pi\in\cal{E}(G,s)$ and $\pi'\in \cal{E}(G',s')$ for
some semisimple elements $s \in G^*$ and $s'\in (G^{\prime *})$. Write $\cal{L}_s(\pi)=\prod\pi[a]$
and $\cal{L}_{s'}(\pi')=\prod\pi'[a]$. Suppose that $q$ is large enough so that the main result in \cite{S} holds. Assume that the unipotent part $\pi[1]$ (resp. $\pi'[1]$) of $\pi$ (resp. $\pi'$) is an irreducible representation of $\UU_k\fq$ (resp. $\UU_{k'}\fq$).
 Then $\pi \otimes \pi'$ occurs in $\omega_{n,n'}$ if and only if the following conditions hold:
\begin{itemize}

\item $\pi[a]=\pi'[a]$ for each $[a]\ne 1$;

\item $\pi[1]\otimes \pi'[1]$ occurs in $\omega_{k,k'}$.
\end{itemize}
That is, the following commutative diagram:
\[
\setlength{\unitlength}{0.8cm}
\begin{picture}(20,5)
\thicklines
\put(5.8,4){$\pi$}
\put(5.2,2.6){$\cal{L}_s$}
\put(13.8,2.6){$\cal{L}_{s'}$}
\put(9.4,4.2){$\Theta$}
\put(4.3,1){$\pi[1]\otimes\prod_{[a]\ne 1}\pi[a]$}
\put(13.4,4){$\pi'$}
\put(11.8,1){$\pi'[1]\otimes\prod_{[a]\ne 1}\pi'[a]$}
\put(8.5,1.5){$\Theta\otimes\prod_{[a]\ne 1}\rm{id}$}
\put(6,3.6){\vector(0,-1){2.1}}
\put(13.5,3.6){\vector(0,-1){2.1}}
\put(8.2,1.1){\vector(2,0){3}}
\put(8.2,4){\vector(2,0){3}}
\end{picture}
\]

\end{theorem}

\subsection{Moment maps}\label{sec4.2}
In this subsection, we recall the main result in \cite{GZ}, which describes the relationship between Theta correspondence and generalized Whittaker models.
Recall the
moment maps:
\[
\setlength{\unitlength}{0.8cm}
\begin{picture}(20,5)
\thicklines
\put(8.8,4.2){$\textrm{Hom}(V, V')$}
\put(6,1){$\fg$}
\put(13.8,1){$\fg'$}
\put(8.5,3.5){\vector(-1,-1){1.5}}
\put(11.5,3.5){\vector(1,-1){1.5}}
\put(12.5,2.8){$\phi'$}
\put(7.3,2.8){$\phi$}
\end{picture}
\]
where $\phi(T) = T^*T$ and $\phi'(T) = T T^*$.
Let $\fg(V)$ and $\fg(V')$ be the Lie algebras of $\UU(V)$ and $\UU(V')$, respectively.
For a given $\mathfrak{sl}_2$-triple $\gamma= \{X, H, Y \} \subset \fg(V)$, let $\fg_\gamma:=\rm{span}\{X, H, Y \}\subset  \fg(V)$. We have the decompositions
\[
V = \bigoplus_{ k\in\bb{Z}} V_k
\]
and
\[
V = \bigoplus_{ j} V^{\gamma,t_j}
\]
where $V_k = \{v \in V | Hv = kv\}$, and $V^{\gamma,t_j}$ is a direct sum of irreducible $t_j$-dimensional $\fg_\gamma$ modules.
Similar notations apply for an $\mathfrak{sl}_2$-triple $\gamma' = \{X', H', Y'\} \in \fg(V')$.
\begin{definition}
Let $\gamma\subset  \fg(V)$, $\gamma'\subset  \fg(V')$
be two $\mathfrak{sl}_2$-triples, and $T \in \textrm{Hom}(V, V ')$. We
say that $T$ lifts $\gamma$ to $\gamma'$
if
 \begin{itemize}
 \item $ T \in \textrm{Gen Hom}(V, V ')$;
 \item $\phi(T) = X$ and $\phi'(T) = X'$;
 \item $T(V_k)\subset V_{k+1}'$ for all $k$,
 \end{itemize}
 where
 \[
 \textrm{Gen Hom}(V, V') = \{T \in \textrm{Hom}(V, V ')|\textrm{ Ker$(T)$ is non-degenerate}\}.
\]
\end{definition}
We set
\[
 \co_{\gamma,\gamma'}:=\{T \in \textrm{Hom}(V, V')|\textrm{ $T$ lifts $\gamma$ to $\gamma'$}\}.
 \]
According to \cite[Lemma 3.4]{Z}, for any $\mathfrak{sl}_2$-triple $\gamma'=\{X',H',Y'\}\subset \fg(V')$ with $X'$ in the image of $\phi'$, there is a unique
$\mathfrak{sl}_2$-triple $\gamma=\{X,H,Y\}\subset \fg(V)$ such that $\co_{\gamma,\gamma'}$ is non-empty. Moreover, $\co_{\gamma,\gamma'}$ is a single $M(V)_X\times M(V')_{X'}$-orbit where $M(V)_X$ and $M(V')_{X'}$ are  stabilizer groups of $X$ and $X'$ in $\UU(V)$ and $\UU(V')$, respectively.

\begin{definition} We will say that $\gamma$ (resp. $\co$) is the generalized descent of $\gamma'$ (resp. $\co'$) if there exists $T$ lifting $\gamma$ to $\gamma'$. We write
\[
\co = \bigtriangledown^{\rm{gen}}_{V',V}(\co').
\]
When $T \in \co_{\gamma,\gamma'} $ is injective, we call $\gamma$
(resp. $\co$) the descent of $\gamma'$ (resp. $\co'$), and write
\[
\co = \bigtriangledown_{V',V}(\co').
\]
Note that The injective of $T $ implies that
\[
\rm{dim}V\le\rm{dim}V'.
\]

If $\UU(V)\cong\UU_n\fq$ and $\UU(V')\cong\UU_m\fq$, then for simplicity notations, we write  $\bigtriangledown_{m,n}$ (resp. $\bigtriangledown^{\rm{gen}}_{m,n}$) instead of $\bigtriangledown_{V',V}$ (resp. $\bigtriangledown^{\rm{gen}}_{V',V}$).
\end{definition}

The explicit description of $\bigtriangledown^{\rm{gen}}_{V',V}$ for the local fields case is given in \cite[section 3]{GZ}. In the finite fields case, one can describe $\bigtriangledown^{\rm{gen}}_{V', V}$ in a similar way. Here we describe $\bigtriangledown^{\rm{gen}}_{V',V}$ for unitary groups as follows.

\begin{proposition}
    Assume that $\UU(V)\cong\UU_n\fq$ and $\UU(V')\cong\UU_m\fq$. Let $\co'$ be an $F$-rational nilpotent orbit of $\UU(V')$. Assume that $\co$ corresponds to the partition $\lambda'=((a_1+1)^{k_1},\cdots, (a_l+1)^{k_l},2^{j'},1^s)$. Then $\co = \bigtriangledown_{m,n}(\co')$ is the $F$-rational nilpotent orbit of $\UU(V)$ corresponding to the partition $\lambda=(a_1^{k_1},\cdots, a_l^{k_l},1^{j})$ with $j\ge j'$.
\end{proposition}

\begin{definition}
 Let $\gamma$ and $\gamma^\prime$ be as before, and $T \in \co_{\gamma,\gamma^\prime}$,
 $V_{\gamma,\gamma^\prime}:= \rm{Ker}(T)$ and $V'_{\gamma,\gamma^\prime}$ be the space of $\gamma^\prime$-invariants. Let $L$ and $L^\prime$ be the isometry groups of $V_{\gamma,\gamma^\prime}$ and $V'_{\gamma,\gamma^\prime}$, respectively.
We define a group homomorphism
\[
\alpha = \alpha_T : M(V')_{X'} \to M(V)_X
\]
by setting $\alpha_T(m')$  to be the identity map on $\rm{Ker}(T)$, and equal  $T^{-1}m'T$ on $(\rm{Ker}T)^\perp$, for $m'\in M(V')_{X'}$.
We have a direct product
\[
M(V')_{X'} = M(V')_{X,X'} \times L',\textrm{ where }M(V')_{X,X'}=\prod_{j=1}^r  \UU((V^{\prime})^{\gamma',t_{j+1}}_{{t_j}})\times \UU(U_1),
\]
where $V^{\gamma,t_j}_k=V^{\gamma,t_j}\cap V_k$, and $U_1\cong (V')^{\gamma,2}_1$ as non-degenerate $\epsilon$-Hermitian space.
Similarly, we have
\[
M(V)_X = M(V)_{X,X'} \times L,\textrm{ where }M(V)_{X,X'}=\prod_{j=1}^r  \UU((V)^{\gamma,t_{j+1}}_{{t_j}})\times \UU(U),
\]
where $U$ is defined in \cite[P.10]{Z}.
\end{definition}

\begin{theorem} (a finite fields analogy of \cite[Theorem 3.7]{Z}\label{gz})
Let $(G, G')$ be a dual pair. Let $\co$ (resp. $\co'$) be an $F$-rational nilpotent orbit in the Lie algebra $\fg^F$ of $G^F$ (resp. $\fg^{\prime F}$ of $G^{\prime F}$).

(i) Assume that $\co'$ is contained in the image of $\co$ under the moment maps. The associated stabilizer groups $M_X$ and $M'_{X'}$ are of the form
\[
M_X \subset M_{X,X'} \times L, M'_{X'} = M_{X,X'} \times L',
\]
for some reductive dual pair $(L, L')$ of the same type as $(G,G')$. Then
\[
\rm{Wh}_{\co',\tau'} (\Theta(\pi))\cong \rm{Wh}_{\co,\Theta(\tau')^\vee} (\pi^\vee)
\]
as $M_{X,X'}$-modules. Here $\Theta(\tau')$ is the theta lifting of $\tau'$ with respect to the dual pair $(L, L')$.

(ii) If $\co'$ is not in the image of $\co$ under the moment maps, then
\[
\rm{Wh}_{\co'} (\pi) = 0.
\]

\end{theorem}

To prove our main theorem, we need to consider the ``generalized Whittaker model of generalized Whittaker model of an irreducible representation''.
\begin{proposition}\label{key}
Let $G=\UU_n$ and $G'=\UU_m$ with $m>n$, and $\co=\co_{(\ell,1,\cdots,1)}$ be an $F$-rational orbit of nilpotent elements in $\fg^F$. Let $G^*=\UU_{n-\ell}$ and $G^{\prime *}=\UU_{m-\ell-1}$, and $\co^*=\co_{(d_2,\cdots,d_s)}$ be an $F$-rational orbit of nilpotent elements in $\fg^{*F}$. Let
\[
\co^{\prime *}=\bigtriangledown_{{m-\ell-1},{n-\ell}}(\co^*)
\]
and
\[
\co^{\prime }=\bigtriangledown_{{m},{n}}(\co).
\]
For an irreducible representation $\pi\in\rm{Irr}(G^F)$, we have
\[
\rm{Wh}_{\co^{\prime *}}(\rm{Wh}_{\co'}(\Theta_{n,m}(\pi)))\cong \rm{Wh}_{\co^*}(\rm{Wh}_{\co}(\pi^\vee)).
\]
\end{proposition}
\begin{proof}
 We regard $\rm{Wh}_{\co}(\pi^\vee)$ as a representation of $\UU_{n-\ell}$ and apply Theorem \ref{gz}. Assume that  we have the following decomposition
 \[
 \rm{Wh}_{\co}(\pi^\vee)=\bigoplus_i\pi_i.
 \]
 Then we have
\[
\rm{Wh}_{\co^*}(\rm{Wh}_{\co}(\pi^\vee))=
\rm{Wh}_{\co^*}\left(\bigoplus_i\pi_i\right)
\cong \rm{Wh}_{\co^{\prime *}}\left(\bigoplus_i\Theta_{n-\ell,m-\ell-1}(\pi_i)\right).
\]
For each irreducible constituent $\tau$ of $\bigoplus_i\Theta_{n-\ell,m-\ell-1}(\pi_i)$,  we have
\[
\begin{aligned}
&\left\langle\tau,\bigoplus_i\Theta_{n-\ell,m-\ell-1}(\pi_i)\right\rangle_{\UU_{m-\ell-1}\fq}\\
=&\left\langle \Theta_{m-\ell-1,n-\ell}(\tau),\bigoplus_i\pi_i\right\rangle_{\UU_{n-\ell}\fq}\\
=&\left\langle \Theta_{m-\ell-1,n-\ell}(\tau),\rm{Wh}_{\co'}(\pi)\right\rangle_{\UU_{n-\ell}\fq}.
\end{aligned}
\]
We take the contragradient representations on the RHS of the above equation and use Theorem \ref{gz} again,
\[
\langle \Theta_{m-\ell-1,n-\ell}(\tau),\rm{Wh}_{\co'}(\pi)\rangle_{\UU_{n-\ell}\fq}=\langle \Theta_{m-\ell-1,n-\ell}(\tau^\vee),\rm{Wh}_{\co}(\pi^\vee)\rangle_{\UU_{n-\ell}\fq}=\langle\tau,\rm{Wh}_{\co'}(\Theta_{n,m}(\pi))\rangle_{\UU_{n-\ell}\fq}
\]
So
\[
\bigoplus_i\Theta_{n-\ell,m-\ell-1}(\pi_i)\cong \rm{Wh}_{\co'}(\Theta_{n,m}(\pi)),
\]
which completes the proof.

\end{proof}

\section{Gan-Gross-Prasad problem and descents for finite unitary groups}\label{sec5}

\subsection{Construction of descents}\label{sec5.1}

Roughly speaking, the Gan-Gross-Prasad problem concerns the generalized Whittaker model of an irreducible representation $\pi\in\rm{Irr}(G^F)$ on an $F$-rational nilpotent orbits $\co_\lambda$ with $\lambda=(\ell,1,1,\cdots,1)$.
According to the parity of $\ell$, the generalized Whittaker model is called the Bessel or Fourier-Jacobi model. In this paper, we will focus on the finite unitary case, and we realize unitary groups as
\[
G^F=\UU_n\fq:=\left\{x\in \GGL_n(\overline{\bb{F}}_q)|\ F(x)=x\right\}
\]
where
\[
F:(x_{i,j})\to (x^q_{n+1-j,n+1-j})^{-1}.
\]
Now, we give the explicit construction of $\rm{Wh}_{\co_{(\ell,1,\cdots,1)}}(\pi)$ for above unitary groups.

(1) {\bf Fourier-Jacobi model}. Assume that $\ell$ is even, and let $k=\frac{\ell}{2}$. Let $P_\ell=M_\ell N_\ell$ be the parabolic subgroup of $\rm{U}_{n}$ with Levi factor $M_\ell\cong \rm{Res}_{\bb{F}_{q^2}/\Fq}\GGL_1^{k}\times \UU_{n-\ell}$ and its unipotent radical $N_\ell$ can be written in the form
\[
N_{\ell}=\left\{n=
\begin{pmatrix}
z & y & x\\
0 & I_{n-\ell} & y^*\\
0 & 0 & z^*
\end{pmatrix}
: z\in U_{k}
\right\},
\]
where $U_k$ is the subgroup of unipotent upper triangular matrice of $\rm{Res}_{\bb{F}_{q^2}/\Fq}\GGL_k$, and the superscript $*$ sends $z=(z_{i,j})$ (resp. $y=(y_{i,j})$) to $z^*=(z_{k+1-j,k+1-i}^{-q})$ (resp. $y^*=(y_{n-\ell+1-j,k-i}^{-q})$).

Consider the $F$-rational nilpotent orbits $\co_\lambda$ corresponding to $\lambda=(\ell,1,1,\cdots,1)=(2k,1,\cdots,1)$. In this setting, $G_{\ge 1}= N_\ell$ and $M_X=\UU_{n-\ell}$. Let
\[
H:=\UU_{n-\ell}\ltimes N_{\ell}=M_X\ltimes G_{\ge 1}
\]
and
 \[
m(\pi,\pi'):=\rm{dim}\rm{Hom}_{H(\Fq)}(\pi'\otimes \omega_n, \pi)
 \]
with $\pi\in\rm{Irr}(\UU_{n}\fq)$ and $\pi'\in\rm{Irr}(\UU_{n-\ell}\fq)$.

 Define the $k$-th {\sl Fourier-Jacobi quotient} of $\pi$ to be
\begin{equation}\label{lfd}
\CQ_{k}^\rm{FJ}(\pi):=\rm{Wh}_{\co_\lambda}(\pi)=\rm{Wh}_{\co_{(2k,1,\cdots,1)}}(\pi)
\end{equation}
viewed as a representation of $\UU_{n-\ell}(\Fq)$. Define the {\sl first occurrence index of of Fourier-Jacobi quotient} $k_0:=k_{0}^\rm{FJ}(\pi)$ of $\pi$ to be the largest nonnegative integer $k_0\leq n$ such that $\CQ^\rm{FJ}_{k_0}(\pi)\neq 0$. The $k_0$-th Fourier-Jacobi quotient of $\pi$ is called the {\sl first Fourier-Jacobi descent} of $\pi$, denoted by
\begin{equation}\label{fjd}
\CD^\rm{FJ}_{k_0}(\pi) :=\CQ^\rm{FJ}_{k_0}(\pi) \ (\textrm{or simply }\CD^\rm{FJ}(\pi)).
\end{equation}

(2) {\bf Bessel model}. Assume that $\ell$ is odd, and let $k=\frac{\ell-1}{2}$. Let $V_n$ be an n-dimensional space over $\mathbb{F}_{q^2}$ with a nondegenerate Hermitian form $(,)$. Assume that  $V_n$ has a decomposition
\[
V_n=X+V_{n-2k}+X^\vee
\]
where $V_{n-2k}$ is a Hermitian subspace of $V_n$ with dimension $n-2k$, and $X+X^\vee=V_{n-\ell}^\perp$ is a polarization. Let $\{e_1,\ldots, e_\ell\}$ be a basis of $X$, $\{e_1',\ldots, e_k'\}$ be the dual basis of $X^\vee$, and
$X_i=\rm{Span}_{\mathbb{F}_{q^2}}\{e_1,\ldots, e_i\}$, $i=1,\ldots, k$. Let $P_\ell=M_\ell N_\ell$ be the parabolic subgroup of $\UU_n$ stabilizing the flag
\[
X_1\subset\cdots\subset X_k,
\]
so that its Levi component is $M_\ell\cong \rm{Res}_{\bb{F}_{q^2}/\Fq}\GGL_1^{k}\times \UU_{n-\ell+1}$ and its unipotent radical $N_\ell$ can be written in the form
\[
N_{\ell}=\left\{n=
\begin{pmatrix}
z & y & x\\
0 & I_{n-\ell+1} & y^*\\
0 & 0 & z^*
\end{pmatrix}
: z\in U_{k}
\right\},
\]
where $U_k$ is the subgroup of unipotent upper triangular matrices of $\rm{Res}_{\bb{F}_{q^2}/\Fq}\GGL_k$. Pick up an anisotropic vector $v_0\in V_{n-2k}$ and define a generic character $\psi_{\ell, v_0}$ of $N_{\ell}^F$ by
\[
\psi_{\ell, v_0}(n)=\psi\left(\sum^{k-1}_{i=1}z_{i,i+1}+(y_\ell, v_0)\right), \quad n\in N_{\ell}^F,
\]
where $y_\ell$ is the last row of $y$. The stabilizer of $\psi_{\ell, v_0}$ in $M_\ell^F$ is  the unitary group of the orthogonal complement of $v_0$ in $V_{n-2k}$, which will be identified with the unitary group $\UU_{n-\ell+1}(\Fq)$.
Put
\begin{equation}\label{hnu}
H:=\rm{U}_{n-\ell}\ltimes N_{\ell},\quad \nu=\psi_{\ell, v_0}.
\end{equation}
For irreducible representations $\pi$ and $\pi'$ of $\rm{U}_n(\Fq)$ and $\rm{U}_{n-2\ell-1}(\Fq)$ respectively, we set
\[
m(\pi,\pi'):=\dim \rm{Hom}_{H^F}(\pi\otimes\bar{\nu}, \pi')\cong
\rm{Hom}_{\rm{U}_{n-\ell}(\Fq)}(\rm{Wh}_{\co_\lambda}(\pi),\pi').
\]
We simply define the notion of the $k$-th {\sl Bessel quotient} of $\pi$ by
\begin{equation}\label{lbd}
\CQ^\rm{B}_{k}(\pi):=\rm{Wh}_{\co_\lambda}(\pi)=\rm{Wh}_{\co_{(2k+1,1,\cdots,1)}}(\pi),
\end{equation}
viewed as a representation of $\UU_{n-\ell}(\Fq)$. Define the {\sl first occurrence index of Bessel quotient} $k_0:=k_0^\rm{B}(\pi)$ of $\pi$ to be the largest nonnegative integer $k_0<n/2$ such that $\CQ^\rm{B}_{k_0}(\pi)\neq 0$. The $k_0$-th Bessel descent of $\pi$ is called the {\sl first Bessel descent} of $\pi$, denoted
\begin{equation}\label{bed}
\CD^\rm{B}_{k_0}(\pi):=\CQ^\rm{B}_{k_0}(\pi)\ (\textrm{or simply }\CD^\rm{B}(\pi)).
\end{equation}

We set
\[
\CQ_{\ell}(\pi):=\left\{
\begin{array}{ll}
\CQ^\rm{B}_{\frac{\ell-1}{2}}(\pi)  ,&\textrm{if $\ell$ is odd;}\\
\CQ^\rm{FJ}_{\frac{\ell}{2}}(\pi),&\textrm{if $\ell$ is even.}
\end{array}\right.
\]
Define the {\sl essential first occurrence index of descent} or simply descent index $\ell_0:=\ell_0(\pi)$ of $\pi$ to be the largest nonnegative integer such that $\CQ_{\ell_0}(\pi)\neq 0$ and $\CQ_{\ell}(\pi)= 0$ for any $\ell>\ell_0$, i.e. $\ell_0:=\rm{max}\{2k_0^\rm{FJ}(\pi),2k_0^\rm{B}(\pi)+1\}$. The $\ell_0$-th descent of $\pi$ is called the {\sl essential first descent} or simply first descent of $\pi$, denoted
\begin{equation}\label{bd}
\CD_{\ell_0}(\pi):=\left\{
\begin{array}{ll}
\CQ^\rm{B}_{\frac{\ell_0-1}{2}}(\pi)  ,&\textrm{if $\ell_0$ is odd;}\\
\CQ^\rm{FJ}_{\frac{\ell_0}{2}}(\pi),&\textrm{if $\ell_0$ is even.}
\end{array}\right.
\end{equation}

Next, we consider ``the descent of the descent of $\pi$''.  We call a series of irreducible representations $\{\pi_i\}$ for a descent sequence of $\pi$ if
\[
\pi_1\xrightarrow{\ell_1}\pi_2\xrightarrow{\ell_2}\pi_{3}\xrightarrow{\ell_3}\cdots\xrightarrow{\ell_k}\bf{1}
\]
where $\pi_1=\pi$ and $\ell_i=\ell_{0}(\pi_{i})$, and $\pi_i$ appears in $\CD_{\ell_{i-1},  _{i-1}}(\pi_{i-1})$, and the last $\bf{1}$ is the trivial representation of trivial group. We call the array $\ell(\pi):=(\ell_1,\ell_2,\cdots)$ for the descent sequence index of $\pi$ with respect to $\{\pi_i\}$.

\subsection{Calculation of descents}\label{sec5.2}

This subsection aims to recall the descents result for irreducible representations of finite unitary groups in \cite{LW2, LW3, Wang2} and calculate their descent sequence.
\begin{theorem}\label{a1}
Suppose that $q$ is large enough so that the main result in \cite{S} holds. Let $\pi$ be an irreducible representation of $\UU_{n}\fq$ and $\pi'$ be an irreducible representation of $\UU_{m}\fq$ with $n\ge m$.
We have
\begin{equation}\label{eq5.1}
   m(\pi,\pi')  =\prod_{[a]}
 m(\pi[a],\pi'[a])
\end{equation}
 where $\pi[a]$ and $\pi'[a]$ are defined in
 Section \ref{sec2.3}, which are both unipotent representations of finite general linear groups or finite unitary groups, and the multiplicity $m(\pi[a],\pi'[a])$ is defined in Section \ref{sec5.1} for unitary groups and in \cite[p.4]{Wang2} for general linear groups.
   In particular, the irreducible representation $\pi'$ appears in the first descent of $\pi$ if and only if for each $[a]$, $\pi'[a]$ appears in the first descent of $\pi[a]$.

\end{theorem}

\begin{proof}
If $n-m$ is odd, in the Bessel case, the first part of this theorem follows from \cite[Theorem 1.3]{Wang2}.

Assume that $n-m$ is even, it is the Fourier-Jacobi. We can reduce the multiplicity into the basic case as in \cite[Proposition 6.5]{LW2}. We have
\begin{equation}\label{bs1}
m(\pi,\pi')=\langle \pi\otimes \omega_n,\rm{Ind}^{\UU_n}_{P}(\tau\otimes \pi')\rangle_{\UU_n\fq},
\end{equation}
where $P$ is an $F$-stable parabolic subgroup of $\UU_n$ with Levi factor $\rm{Res}_{\bb{F}_{q^2}/\Fq}\GGL_\ell\times \UU_{n'}$, and $\tau$ is an irreducile cuspidal representation of $\GGL_\ell(\bb{F}_{q^2})$.
In \cite[Proposition 6.5]{LW2}, we assume that $\pi$ is unipotent. However, in the proof of \cite[Proposition 6.5]{LW2}, the assumption of $\pi$ is only used to get
\begin{equation}\label{bs2}
\langle \pi,\rm{Ind}^{\UU_n}_{P}(\tau\otimes \pi'')\rangle_{\UU_n\fq}=0
\end{equation}
for any $\pi''\in \rm{Irr}(\UU_m\fq)$. Assume that $\pi\in \cal{E}(\UU_n\fq,s)$ and $\tau\in \cal{E}(\GGL_k(\bb{F}_{q^2}),t)$. If we additionally assume that $t$ and $s$ have no common eigenvalues, then we still have (\ref{bs2}), and so is (\ref{bs1}). This kind of $t$ always exists because we assume that the finite field $\Fq$ is large enough.

 We now turn to calculate the RHS of (\ref{bs1}). Recall that unipotent representations of unitary groups are parametrized by partitions. Assume that $\cal{L}_\pi([1])=\lambda=(\lambda_1,\cdots,\lambda_h)$ and $\cal{L}_{\pi'}([1])=\lambda'=(\lambda'_1,\cdots,\lambda'_{h'})$. Let $d=\sum\lambda_i$ and $d'=\sum\lambda_i'$. Consider an irreducible representation $\sigma\in\cal{E}(\UU_{2n}\fq,s'')$ satisfying the follows conditions:
\begin{itemize}

\item $\cal{L}_{s''}(\sigma)=\prod_{[a]}\sigma[a]$;

\item $\sigma[a]=\pi[a]$ for each $[a]\ne 1$;

\item $\cal{L}_{\sigma}([1])=(n,\lambda_1,\cdots,\lambda_h)$.
\end{itemize}
It follows \cite[Corollary 3.5]{LW3} that $\Theta_{n+d,d}(\sigma[1])=\pi[1]$. Then by Theorem \ref{p1}, $\Theta_{2n,n}(\sigma)=\pi$.
Consider the see-saw diagram
\[
\setlength{\unitlength}{0.8cm}
\begin{picture}(20,5)
\thicklines
\put(6.8,4){$\UU_{n}\times \UU_{n}$}
\put(7.3,1){$\UU_{n}$}
\put(12.2,4){$\UU_{2n+1}$}
\put(11.9,1){$\UU_{2n}\times \UU_1$}
\put(7.7,1.5){\line(0,1){2.1}}
\put(12.8,1.5){\line(0,1){2.1}}
\put(8,1.5){\line(2,1){4.2}}
\put(8,3.7){\line(2,-1){4.2}}
\end{picture}
\]
By the see-saw identity, one has
\[
\begin{aligned}
&\langle  \pi\otimes \omega_n,\rm{Ind}^{\UU_n}_P(\tau\otimes \pi')\rangle_{\UU_n\fq}\\
=& \langle \Theta_{2n,n}(\sigma)\otimes \omega_n,\rm{Ind}^{\UU_n}_P(\tau\otimes \pi')\rangle_{\UU_n\fq}\\
=& \langle \sigma,\Theta_{n,2n+1}(\rm{Ind}^{\UU_n}_P(\tau\otimes \pi'))\rangle_{\UU_{2n}\fq}.\\
\end{aligned}
\]

We now calculate $\Theta_{n,2n+1}(\rm{Ind}^{\UU_n}_P(\tau\otimes \pi'))$. According to Proposition \ref{theta}, we have
\[
\Theta_{n,2n+1}(\rm{Ind}^{\UU_n}_P(\tau\otimes \pi'))=\rm{Ind}^{\UU_{2n+1}}_{P'}(\tau\otimes \Theta_{m,m+n+1}(\pi')),
 \]
 where $P'$ is the an $F$-stable parabolic subgroup of $\UU_{2n+1}$ with Levi factor $\rm{Res}_{\bb{F}_{q^2}/\Fq}\GGL_\ell\times \UU_{m+n+1}$. Applying  Theorem \ref{p1} on $\Theta_{m,m+n+1}(\pi')$, we know that each irreducible component of $\Theta_{m,m+n+1}(\pi')$ is in the same Lusztig series, and we denote the Lusztig series by $\cal{E}(\UU_{m+n+1}\fq,s')$. Moreover, by Theorem \ref{p1} and \cite[Proposition 3.4]{LW3}, $\Theta_{m,m+n+1}(\pi')$ can be described explicitly as follows:
\[
\Theta_{m,m+n+1}(\pi')=\oplus_{\mu'}\sigma'_{\mu'}
\]
where $\sigma'_{\mu'}$ is an irreducible such that
\begin{itemize}

\item $\cal{L}_{s''}(\sigma'_{\mu'})=\prod_{[a]}\sigma'_{\mu'}[a]$;

\item $\sigma'_{\mu'}[a]=\pi'[a]$ for each $[a]\ne 1$;

\item $\cal{L}_{\sigma'_{\mu'}}([1])=\mu'$,
\end{itemize}
and the sum runs over partitions $\mu'$ such that $\pi_{\mu'}$ appears in $\Theta_{d',n+1+d'}(\pi'[1])$.
Then we have
\[
\begin{aligned}
&\langle \sigma,\Theta_{n,2n+1}(\rm{Ind}^{\UU_n}_P(\tau\otimes \pi'))\rangle_{\UU_{2n}\fq}\\\
=&\langle \sigma,\oplus_{\mu'}   \rm{Ind}^{\UU_{2n+1}}_{P'}(\tau\otimes \sigma_{\mu'})\rangle_{\UU_{2n}\fq}\\
=&\sum_{\mu'}  \langle \sigma, \rm{Ind}^{\UU_{2n+1}}_{P'}(\tau\otimes \sigma_{\mu'})\rangle_{\UU_{2n}\fq}\\
=&\sum_{\mu'} m(\sigma,\sigma_{\mu'})
\end{aligned}
\]
Here $m(\sigma,\sigma_{\mu'})$ is the multiplicity in the Bessel case, so we have
\[
\sum_{\mu'}m(\sigma,\sigma_{\mu'})=\sum_{\mu'}\prod_{[a]}m(\sigma[a],\sigma_{\mu'}[a])=m(\pi_{(n,\lambda_1,\cdots,\lambda_h)},\oplus_{\mu'}\pi_{\mu'})\cdot\prod_{[a]\ne 1}m(\pi[a],\pi'[a]).
\]
With a similar see-saw argument, we get
\[
m(\pi[1],\pi'[1])=m(\pi_{(n,\lambda_1,\cdots,\lambda_h)},\Theta_{d',n+1+d'}(\pi_{\lambda'}))=m(\pi_{(n,\lambda_1,\cdots,\lambda_h)},\oplus_{\mu'}\pi_{\mu'}).
\]
So
\[
m(\pi,\pi')=\sum_{\mu'}m(\sigma,\sigma_{\mu'})=m(\pi_{(n,\lambda_1,\cdots,\lambda_h)},\oplus_{\mu'}\pi_{\mu'})\cdot\prod_{[a]\ne 1}m(\pi[a],\pi'[a])=\prod_{[a]}m(\pi[a],\pi'[a]).
\]

Assume that $\pi'$ appears in the first descent of $\pi$, and there exist $[a^*]$ such that $\pi[a^*]$ does not appear in the first descent of $\pi[a^*]$. By our decomposition of $m(\pi,\pi')$, we know that $m(\pi[a^*],\pi'[a^*])\ne 0$. Let $\sigma[a^*]$ be an irreducible representation in the first descent of $\pi[a^*]$. Then the group of $\sigma[a^*]$' is smaller then the group of $\pi'[a^*]$. Let $\sigma'$ be an irreducible representation such that its image of Lusztig correspondence $\prod_{[a]}\sigma'[a]$ satisfies the following conditions:
\begin{itemize}

\item $\sigma'[a]=\pi'[a]$ for each $[a]\ne [a^*]$;

\item $\sigma'[a]=\sigma[a^*]$.
\end{itemize}
Hence we get an irreducible representation $\sigma'$ living in a group smaller than the group of  $\pi'[a]$, and $m(\pi,\sigma')\ne 0$. This is a contradiction with the definition of the first descent. So for each $[a]$, $\pi'[a]$ have to appear in the first descent of $\pi[a]$. The proof of the only if part is similar, which is left to readers.

\end{proof}

In our previous work, we have calculated the descent of unipotent representations of unitary groups in \cite{LW2, LW4}:

\begin{theorem}\label{a2}\cite[Theorem 1.1]{LW2}
Let $\lambda$ be a partition of $n$ into $k$ rows, and $\lambda^{-1}$ be the partition of $n-k$ obtained by removing the first column of $\lambda$. Let $\pi\in\rm{Irr}(\UU_{m}\fq)$. Then
\[
m(\pi_\lambda,\pi')=0
\]
if $m< n-k$. Moreover, for $m=n-k$,
\[
m(\pi_\lambda,\pi')=
\left\{
  \begin{aligned}
 1&\textrm{ if }\pi'=\pi_{\lambda^{-1}};\\
  0&\textrm{ otherwise}.
\end{aligned}
\right.
\]
\end{theorem}
\begin{theorem}\cite[Theorem 1.1]{LW4}\label{a3}
Assume that $n \ge m$. Let $\lambda$ and $\lambda'$ be partitions of $n$ and $m$ respectively. We say that $\lambda$ and $\lambda'$ are  2-transverse if $|\lambda_i-\lambda_i'|\le1$ for every $i$ and $\#\{i|\lambda_i=\lambda'_i=j\}$ is even for any $j>0$.
Then
\[
m(\pi_\lambda, \pi_{\lambda'})=\langle\pi_{^t\lambda}\otimes\pi_{^t\lambda'},\omega_{n,m}\rangle_{\UU_n\fq\times\UU_m\fq}=\left\{
\begin{array}{ll}
1, &  \textrm{if }\ \lambda \textrm{ and } \lambda' \textrm{ are }  2\textrm{-transverse},\\
0, & \textrm{otherwise,}
\end{array}\right.
\]
where $\omega_{n,m}$ is the Weil representation of $\UU_n\fq\times\UU_m\fq$.
\end{theorem}

By the same argument, a similar result holds for general linear groups. This branching problem from $\GGL_n\fq$ to $\GGL_{n-1}\fq$ is also fully analysed in \cite{T}.

\begin{theorem}\label{ggp}
Let $\pi\in \rm{Irr}(\UU_n\fq)$ (resp. $\pi'\in \rm{Irr}(\UU_m\fq)$). Assume that $n\ge m$. Then
\[
m(\pi, \pi')=\left\{
\begin{array}{ll}
1, &  \textrm{if }\ \cal{L}_\pi([a]) \textrm{ and } \cal{L}_{\pi'}([a]) \textrm{ are }  2\textrm{-transverse for every }[a],\\
0, & \textrm{otherwise,}
\end{array}\right.
\]
where $\cal{L}_\pi$ is defined in subsection \ref{sec2.3}.
\end{theorem}
\begin{proof}
It immediately follows from Theorem \ref{a1}, Theorem \ref{a2} and Theorem \ref{a3}.
\end{proof}

\begin{corollary}\label{d1}
Let $G=\UU_n$ and $\pi\in\cal{E}(G^F,s)$. Suppose that
\[
\cal{L}_s(\pi)=\prod_{[a]}\pi{[a]}
\]
 and $\cal{L}_\pi([a])=\lambda[a]$ for each $[a]$. Write $\lambda[a]^t=(\lambda[a]^t_1,\lambda[a]^t_2,\cdots)$.

 (i) Let $\pi'$ be an irreducible representation of $\UU_m\fq$ with $n\ge m$. Assume that
  the image of $\pi'$ under Lusztig correspondence is
\[
\prod_{[a']}\pi'[a'],
\]
and
 \[
\cal{L}_{\pi'}[a]=\left\{
  \begin{aligned}
 &{\lambda[a]^{-1}}&\textrm{ if $a'$ is a eignvalue of the semisimple element $s$ of $\pi$};\\
  &{(1,1,\cdots,1)}&\textrm{ otherwise}.
\end{aligned}
\right.
\]
 Then
\[
m(\pi,\pi')=1.
\]

(ii) The first occurrence index of $\pi$ is
$\ell_0=\sum_{[a]}\#[a]\cdot\lambda[a]^t_1$, and
\[
 \CD_{\ell_0}(\pi)=\pi''
\]
is a irreducible representation of $\UU_{n-\ell_0}\fq$ such that
$
\cal{L}_{\pi''}([a])={\lambda[a]^{-1}}\textrm{ for each }[a].
$\
In particular, for an irreducible representation $\pi$, the descent sequence (resp. the descent sequence index) of $\pi$ is unique. The the descent sequence index $\ell(\pi)=(\ell_1,\ell_2,\cdots)$ is a partition, and $\ell_i=\sum_{[a]}\#[a]\cdot\lambda[a]^t_i$.
\end{corollary}

\section{Lemma in \cite{GRS} and exchanging roots Lemma}
\label{sec6}
In this section, we recall the result for cuspidal representations in \cite{GRS}. And we generalize their results from cuspidal representations of symplectic groups to the general case for unitary groups. In \cite{PW}, we consider the symplectic groups. However, we only deal with the Fourier-Jacobi case in \cite{PW}. This paper will first prove a similar exchanging roots Lemma for the Fourier-Jacobi case. Then we will get the Bessel case by the theta arguments in Section \ref{sec4}.

\subsection{Exchanging roots Lemma}\label{sec6.1}
Let $C\subset G$ be an $F$-stable subgroup of a maximal unipotent subgroup of $G$, and
let $\psi_C$ be a non-trivial character of $C^F$. Assume that there are two unipotent
$F$-stable subgroups $X$, $Y$ , such that the following conditions are satisfied.
 \begin{itemize}
 \item (1)   $X$ and $Y$ normalize $C$;
\item (2)  $X\cap C$ and $Y \cap C$ are normal in $X$ and $Y$ , respectively. The groups $X \cap C\verb|\|X$ and
$Y \cap C\verb|\|Y $ are abelian;
\item (3) $X^F$ and $Y^F$ preserve $\psi_C$ (when acting by conjugation);
\item (4)  $\psi_C$ is trivial on $(X \cap C)^F$ and on $(Y \cap C)^F$;
\item (5)   $[X, Y ] \subset C$;
\item (6) The pairing $(X \cap C)^F\verb|\|X^F \times(Y \cap C)^F\verb|\|Y^F \to \bb{C}^\times$, given by
\[
(x, y) \to \psi_C([x, y]),
\]
is multiplicative in each coordinate, non-degenerate, and identifies $(Y \cap C)^F\verb|\|Y^F$
with the dual of $(X \cap C)^F\verb|\|X^F$ and $(X \cap C)^F\verb|\|X^F$ with the dual of $(Y \cap C)^F\verb|\|Y^F$.
\end{itemize}
We represent the setup above by the following diagram,
\begin{equation}\label{er}
\begin{matrix}
&&A &&\\
&\nearrow&&\nwarrow&\\
B=CY&&&&D=CX\\
&\nwarrow&&\nearrow&\\
&&C&&
\end{matrix}.
\end{equation}
Here, $A = BX = DY = CXY $. Extend the character $\psi_C$ to a character $\psi_B$ of
$B^F$, and to a character $\psi_D$ of $D^F$, by making it trivial on $Y^F$ and on $X^F$.

\begin{lemma}[Exchanging roots Lemma (Lemma 7.1 in \cite{GRS2})]\label{erl}
Let $\pi$ be an irreducible representation of $G^F$. Then
\[
\langle \pi,\psi_{B}\rangle_{B^F}=\sum_{u\in B^F}\pi(u)\psi_B(u^{-1})=\sum_{y\in (Y \cap C)^F\verb|\|Y^F}\sum_{u\in D^F}\pi(uy)\psi_D(u^{-1})
\]
\end{lemma}

We need to follow the lemma to compose a Fourier coefficient with another Fourier coefficient.
The proof of the following lemma is similar to that of Lemma 2.6 in \cite{GRS}, a special case of Exchanging Roots Lemma. But we would like to give direct proof here.
\begin{lemma}\label{ex1}
Let $G=\UU_n$ and $\pi$ be an irreducible representation of $G^F$.
Let $\cal{O}_\lambda=\co_{(\ell,d_2 ,\cdots ,d_s)}$ be an $F$-rational nilpotent orbit of $\fg^F$. Assume that $n$ and $\ell$ are both even. Then
\[
\rm{dim}\rm{Wh}_{\co_{\lambda}}(\pi)=\rm{dim}\rm{Wh}_{\co_{(d_2,\cdots,d_s)}}(\rm{Wh}_{\co_{(\ell,1 ,\cdots, 1)}}(\pi)).
\]

\end{lemma}

\begin{proof}

Let $G_{\ge k}$ (resp. $G_{\ge k}'$, $G_{\ge k}''$) be certain subgroups of unipotent subgroup of $\UU_n$ (resp. $\UU_n$, $\UU_{n-\ell}$) corresponding to $\co_{(\ell,d_2,d_3,\cdots,d_s)}$ (resp. $\co_{(\ell,1^{n-\ell})}$, $\co_{(d_2,d_3,\cdots,d_s)}$) with $k\in\{1,1.5,2\}$.
To prove this lemma, we only need to show
\begin{equation}\label{ex11}
\frac{1}{|G_{\ge 1.5}^F|}\sum_{u\in G_{\ge 1.5}^F}\pi(u)\psi(u^{-1})=\frac{1}{|G_{\ge 1.5}^{\prime F}||G_{\ge 1.5}^{\prime \prime F}|}\sum_{u_2\in G_{\ge 1.5}^{\prime\prime F}}\left(\sum_{u_1\in G_{\ge 1.5}^{\prime F}}\pi(u_1u_2)\psi(u_1^{-1})\right)\psi(u_2^{-1}).
\end{equation}

We prove (\ref{ex11}) by the root exchange technique in \cite[Chapter 7]{GRS2}. Since the exchanging roots process is precisely the same as the proof of Lemma 7.6 in \cite{PW}, we will explain how to pick some unipotent subgroups of $\UU_n$ which are similar to $X, Y, A, B, C, D$ as in Exchanging roots Lemma (Lemma \ref{erl}), and left the exchanging roots process to the reader.

From the structure of $\co_\lambda$, we can pick a representative element $\gamma=\{H,X,Y\}$ in the orbit of $\frak{sl}_2$-triple corresponding to $\co_\lambda$ such that
\[
\rm{exp}tH = \rm{diag}(t^{\ell-1},   \cdots ,   t^{-\ell+1})\in \UU_n\fq
\]
where $t\in \UU_1\fq$, and the powers of $t$ descend.
There is a Weyl group element $w$ which conjugates $\rm{exp}tH$ to the torus
\[
\begin{aligned}
h(t):&=w(\rm{exp}tH)w^{-1} = \rm{diag}(T,T_0,T'),
\end{aligned}
\]
where
\[
T=(t^{\ell-1}, t^{\ell-3}, \cdots ,t),\quad T_0=( t^{d_2-1}, \cdots , t^{-(d_2-1)})\textrm{ and } T'=(t^{-1},\cdots,t^{-(\ell-3)}, t^{-(\ell-1)})
\]
with the powers of $t$ descending.
Similarly, by the structure of $\co_{(d_2, \cdots ,d_s)}$, we can pick a representative element $\gamma''=\{H'',X'',Y''\}$ in the orbit of $\frak{sl}_2$-triple corresponding to $\co_{(d_2, \cdots ,d_s)}$ such that
\[
\rm{exp}tH'' = \rm{diag}(t^{d_2-1},  \cdots , t^{-(d_2-1)})\in \UU_{n-\ell}\fq,
\]
where the powers of $t$ goes down. Let $\fg_j = \{Z \in \fg | \rm{Ad}(h(t))(Z) = jZ\}$ and
\[
\fg = \bigoplus_{j\in\bb{Z}}\fg_{j}.
\]
Denote by
\[
\fg_{\ge i}:=\bigoplus_{j\ge i}\fg_{j}\textrm{ and } \fg_{\le i}:=\bigoplus_{j\le i}\fg_{j}.
\]
They are Lie algebras of a close connect subgroup $G_{\ge i}$ and $G_{\le i}$ of $G$, respectively.

To write down $G_{\ge i}^F$ explicitly, we set
\begin{itemize}
 \item $z_i$ is the number of $t^{j}$ in $\rm{exp}tH''$ such that $j> \ell-1-2i$;
\item $z'_j$ is the number of $t^{i}$ in $\rm{exp}tH''$ such that $i\ge \ell+3-2j$;
\item  $z^1_i$ is the number of $t^{j}$ in $\rm{exp}tH''$ such that $j= \ell-2i$.
\item $z^{\prime 1}_j$ is the number of $t^{i}$ in $\rm{exp}tH''$ such that $i= \ell+2-2j$.
\end{itemize}
Let $M_{n_1\times n_2}(k)$ be the additive group of $n_1\times n_2$ matrices with entries in the field $k$. For $m=(m_{i,j})\in M_{n_1\times n_2}(k)$, we set
$m^*=(m_{n_2+1-j,n_1+1-i}^{-q})\in M_{n_2\times n_1}$. Then
\[
G_{\ge1.5}^F=\left\{\begin{pmatrix}
u&h+e&g\\
&v&h^*+e^*\\
&&u^*
\end{pmatrix}
\begin{pmatrix}
I&&\\
k+f&I&\\
&k^*+f^*&I
\end{pmatrix}\right\}
\]
where
\begin{itemize}
 \item $h$ runs over $M_{\ge2}=\{(\alpha_{i,j})\in M_{\frac{\ell}{2}\times n-\ell}({\bb{F}}_{q^2})\ |\ \alpha_{i,j}=0\textrm{ if }j\le z_i\}$;
\item $u$ runs over the subgroup of unipotent upper triangular matrices of $\rm{Res}_{\bb{F}_{q^2}/\Fq}\GGL_\frac{\ell}{2}$;
\item  $g$ runs over $M_{\frac{\ell}{2}\times\frac{\ell}{2}}\fqq$;
\item $v$ runs over $G_{\ge1.5}^{\prime \prime F}$;
\item $k$ runs over $ M_{\ge2}'=\{(\alpha_{i,j})\in M_{ n-\ell\times \frac{\ell}{2}}\fqq\ |\ \alpha_{i,j}=0\textrm{ if }i> z'_{j}\}$;
\item $e $ runs over $ M_1=\{(\alpha_{i,j})\in M_{\frac{\ell}{2}\times n-\ell}\fqq\ |\ \alpha_{i,j}=0\textrm{ if }j\le \frac{z^1_i}{2}\textrm{ or }j>z_i\}$;
\item $f$ runs over $ M_1'=\{(\alpha_{i,j})\in M_{ n-\ell\times \frac{\ell}{2}}\fqq\ |\ \alpha_{i,j}=0\textrm{ if }i> (\frac{z^{\prime 1}_{j}}{2}+z^{\prime }_{j})\textrm{ or }i\le z^{\prime }_{j}\}$.
\end{itemize}
For $\co_{(\ell,1^{n-\ell})}$, we have
\[
G^{\prime F}_{\ge1.5}=\left\{\begin{pmatrix}
u&h'&g\\
&I&h^{\prime*}\\
&&u^*
\end{pmatrix}\right\}
=
\left\{\begin{pmatrix}
u&h+e&g\\
&I&h^*+e^*\\
&&u^*
\end{pmatrix}
\begin{pmatrix}
I&d&\\
&I&d^{*}\\
&&I
\end{pmatrix}
\right\}
\]
where
\begin{itemize}

\item $h'$ runs over $\{(\alpha_{i,j})\in M_{\frac{\ell}{2}\times n-\ell}\fqq\ |\ \alpha_{\frac{\ell}{2},j}=0\textrm{ if }j\le n-\ell\}$;

\item $d$ runs over $\{(\alpha_{i,j})\in M_{\frac{\ell}{2}\times n-\ell}\fqq\ |\ \alpha_{i,j}=0 \textrm{ if }j> \frac{z^1_i}{2}\}$.
 \end{itemize}

We set
\[
C:=\left\{\begin{pmatrix}
u&h+e&g\\
&v&h^*+e^*\\
&&u^*
\end{pmatrix}\right\},\quad X:=\left\{\begin{pmatrix}
I&d&\\
&I&d^{ *}\\
&&I
\end{pmatrix}
\right\}
\textrm{ and }
Y:=\left\{\begin{pmatrix}
I&&\\
k+f&I&\\
&k^*+f^*&I
\end{pmatrix}\right\},
\]
and let $\frak{c}$, $\frak{x}$ and $\frak{y}$ be their Lie algebra, respectively. Let
\[
B:=CY,\ D:=CX\textrm{ and }A:=CXY.
\]
Note that
 $(Y \cap C)\verb|\|Y=I\textrm{ and }(X \cap C)\verb|\|X=I$. Then
 \[
 \textrm{ the LHS of (\ref{ex11})}=\frac{1}{|B^F|}\sum_{u\in B^F}\pi(u)\psi(u^{-1})
 \]
 and
 \[
  \textrm{ the RHS of (\ref{ex11})}=\frac{1}{|D^F|}\sum_{u\in D^F}\pi(u)\psi(u^{-1}).
 \]
We would like to show that both side of (\ref{ex11}) is equal to
\[
\frac{1}{| C^F||X^F||Y^F|}\sum_{u\in C^F}\sum_{x\in X^F}\sum_{y\in Y^F}\pi(uyx)\psi(u^{-1}).
\]

 We set unipotent subgroups $X_r$ and $Y_r$ and their Lie algebra $\frak{x}_r$ and $\frak{y}_r$ as follow:
\[
X_r=X\cap (G_{\le2-r}\verb|\|G_{<2-r})\textrm { and }
Y_r=Y\cap (G_{\ge r}\verb|\|G_{> r}).
\]
In other words,
\[
\frak{x}_r=\fg_{2-r}\cap\frak{x} \textrm{ and }\frak{y}_r=\fg_{r}\cap\frak{y}.
\]
Note that $X$ and $Y$ are abelian, we have
\[
Y=\prod_{r=1}^{m}Y_r
\textrm{ and }
X=\prod_{r=1}^{m}X_r.
\]
We now turn to explicitly describing elements in $\frak{x}_r$ and $\frak{y}_r$.
Let
$\beta_j$ be the power of $t$ in the $j$-th term of $w(\rm{exp}tH)w^{-1}$, and $\delta(r,j)=\frac{\ell-\beta_j+r-1}{2}$. Let
\[
\cal{X}_r:=\{(i,j)\ |\ e_{ij}\in \mathfrak{x}_r\}\textrm{ and }\cal{Y}_r:=\{(i,j)\ |\ e_{ij}\in \frak{y}_r\}.
\]
To be more precise, we have
\[
\cal{X}_r=\{(i,j)\ |\ \ \frac{\ell}{2}< j\le n,\ i=\delta(r,j)\textrm{ and }i<\frac{\ell}{2}\}
\]
and
\[
\cal{Y}_r=\{(i,j)\ |\ \frac{\ell}{2}< i\le n,\ j=\delta(r,i)+1\textrm{ and }j\le \frac{\ell}{2}\}.
\]
Then $x\in X_r $ (resp. $y\in Y_r $) if and only if
\[
x_r=I+\sum_{(i,j)\in\cal{X}_r} \alpha_{i,j}e_{ij} \ \textrm{ (resp. } y_r=I+\sum_{(i,j)\in\cal{Y}_r} \alpha_{i,j}'e_{ij}\textrm{)}
\]
with $\alpha_{i,j}\in \Fqq$ (resp. $\alpha'_{i,j}\in \Fqq$).
By direct calculation, we have $[X_r,Y_r]\subset G_{> 2}\verb|\|G_{\ge 2}$, and $[X_r,Y_{r'}]\subset G_{> 2}$ for $r'>r$. So
$
\psi([x_r,y_r'])=1
$
and
\[
\psi([x_r,y_r])=\psi(\sum_{j=n_1+1}^{n}\alpha_{\delta(r,j),j}\cdot\alpha'_{j,\delta(r,j)+1}).
\]
Therefore the pairing $I\verb|\|X_r^F \times I\verb|\|Y_r^F \to \bb{C}^\times$, given by
\[
(x_r, y_r) \to \psi([x_r, y_r]),
\]
is multiplicative in each coordinate, non-degenerate. We identify $I\verb|\|Y_r^F$
with the dual of $I\verb|\|X_r^F$, and $I\verb|\|X_r^F$ with the dual of $I\verb|\|Y_r^F$. In other words, the pair of groups $(X_r, Y_r)$ satisfies the condition (5)-(6) in the Exchanging Roots Lemma.

We set
\[
\begin{aligned}
\phi_{m+1}(g):=\frac{1}{|D^F|}\sum_{u\in D^F}\pi(ug)\psi(u^{-1}).
\end{aligned}
\]
We inductively define
\[
\phi_{r}(g):=\sum_{y_{r}\in Y_{r}^F}\phi_{r+1}(y_{r}g).
\]
In particular,
\[
\frac{1}{|Y^F|}\phi_{1}(1)=\frac{1}{|C^F||X^F||Y^F|}\sum_{u\in C^F}\sum_{x\in X^F}\sum_{y\in Y^F}\pi(uxy)\psi(u^{-1}).
\]
With similar exchanging roots argument in the proof of Lemma 7.6 in \cite{PW}, we know that
\[
\frac{1}{|Y_{r}^F|}\phi_{r+1}(1)=\phi_{r}(1).
\]
Therefore, we have
\begin{equation}\label{ex111}
\begin{aligned}
\textrm{ the RHS of (\ref{ex11})}=\ &\frac{1}{|D^F|}\sum_{u\in D^F}\pi(u)\psi(u^{-1})\\
 =\ &\phi_{m+1}(1)\\
 =\ &\frac{1}{|Y^F|}\phi_{1}(1)\\
 =\ &\frac{1}{|C^F||X^F||Y^F|}\sum_{u\in C^F}\sum_{x\in X^F}\sum_{y\in Y^F}\pi(uxy)\psi(u^{-1}).
 \end{aligned}
\end{equation}

Similarly, for the LHS of (\ref{ex11}), we set
\[
\begin{aligned}
\phi'_{1}(g):=\frac{1}{|B^F|}\sum_{u\in B^F}\pi(ug)\psi(u^{-1}).
\end{aligned}
\]
We inductively define
\[
\phi'_{r}(g):=\sum_{x_{r}\in X_{r}^F}\phi_{r-1}(x_{r}g).
\]
With the same exchanging roots argument, we have
\begin{equation}\label{ex112}
\textrm{ the LHS of (\ref{ex11})}=\frac{1}{|C^F||X^F||Y^F|}\sum_{u\in C^F}\sum_{x\in X^F}\sum_{y\in Y^F}\pi(uxy)\psi(u^{-1}).
\end{equation}
Then (\ref{ex11}) follows immediately from (\ref{ex111}) and (\ref{ex112}).
\end{proof}

\begin{lemma}\label{ex2}
Let $G=\UU_n$ and $\pi$ be an irreducible representation of $G^F$.
Let $\cal{O}_\lambda=\co_{(\ell,d_2 ,\cdots ,d_s)}$ be an $F$-rational nilpotent orbit of $\fg^F$. Assume that $\ell$ is odd. Then
\[
\rm{dim}\rm{Wh}_{\co_{\lambda}}(\pi)=\rm{dim}\rm{Wh}_{\co_{(\ell,1 ,\cdots, 1)}}(\rm{Wh}_{\co_{(d_2,\cdots,d_s)}}(\pi)).
\]

\end{lemma}
\begin{proof}
Let $m=2n$ and $G'=\UU_m$. Let  $\cal{O}'=\co_{(\ell+1,d_2+1 ,\cdots ,d_s+1,1,\cdots,1)}$ be an $F$-rational nilpotent orbit of $\fg^{\prime F}$. Consider the dual pair $(G',G)=(\UU_m,\UU_n)$. Then
\[
 \co_{\lambda}=\bigtriangledown_{m,n}(\co').
\]
By Proposition \ref{key}, we have
\[
\rm{Wh}_{\co^{**}}(\rm{Wh}_{\co^{*}}(\Theta_{n,m}(\pi^\vee)))\cong \rm{Wh}_{\co_{(d_2,\cdots,d_s)}}(\rm{Wh}_{\co_{(\ell,1 ,\cdots, 1)}}(\pi)),
\]
where
\[
\co^*=\bigtriangledown_{m,n}(\co_{(\ell,1 ,\cdots, 1)})=\co_{(\ell+1,1,\cdots,1)}
\]
and
\[
\co^{**}=\bigtriangledown_{{m-\ell-1},{n-\ell}}(\co_{(d_2,\cdots,d_s)})=\co_{(d_2+1,\cdots,d_s+1,1,\cdots,1)}.
\]
Note that $m$ and $\ell+1$ are both even.
By Lemma \ref{ex1}, we have
\[
\rm{dim}\rm{Wh}_{\co^{**}}(\rm{Wh}_{\co^{*}}(\Theta_{n,m}(\pi^\vee)))=\rm{dim}\rm{Wh}_{\co'}(\Theta_{n,m}(\pi^\vee)).
\]
It follows from Theorem \ref{gz} that
\[
\rm{Wh}_{\co'}(\Theta_{n,m}(\pi^\vee))\cong \rm{Wh}_{\co}(\pi).
\]
So
\[
\begin{aligned}
\rm{dim}\rm{Wh}_{\co_{\lambda}}(\pi)=&\rm{dim}\rm{Wh}_{\co'}(\Theta_{n,m}(\pi^\vee))\\
=&\rm{dim}\rm{Wh}_{\co^{**}}(\rm{Wh}_{\co^{*}}(\Theta_{n,m}(\pi^\vee)))\\
=&\rm{dim}\rm{Wh}_{\co^{**}}(\rm{Wh}_{\co^{*}}(\Theta_{n,m}(\pi^\vee))).
\end{aligned}
\]
\end{proof}

\begin{lemma}\label{ex3}
Assume that $G=\UU_n$, and $\pi$ is an irreducible representation of $G^F$.
Let $\cal{O}_\lambda=\co_{(\ell,d_2 ,\cdots ,d_s)}$ be an $F$-rational nilpotent orbit of $\fg^F$. Then
\[
\rm{dim}\rm{Wh}_{\co_{\lambda}}(\pi)=\rm{dim}\rm{Wh}_{\co_{(d_2,\cdots,d_s)}}(\rm{Wh}_{\co_{(\ell,1 ,\cdots, 1)}}(\pi)).
\]

\end{lemma}
\begin{proof}
We have proved this lemma for odd $\ell$ in Lemma \ref{ex2} and for both $n$ and $\ell$ even in Lemma \ref{ex1}.
We only need to prove the case where $\ell$ is even, and $n$ is odd. This case follows from the odd $\ell$ case and similar argument in the proof of Lemma \ref{ex2}.
\end{proof}
\section{Proof of main results}\label{sec7}

In this section, we prove our main theorems. Theorem \ref{am} contains a explicit description of the wavefront set and a multiplicity one result. We mention that the multiplicity one result follows from that fact: the first descent of an irreducible representation is also an irreducible representation.

\begin{theorem}\label{am} Assume that $G=\UU_n$ is a unitary group over $\Fq$, and  $\pi\in \rm{Irr}(G^F)$. Let $\ell(\pi)=(\ell_1,\ell_2,\cdots)$ be the descent index sequence of $\pi$.

(i) We have
\[
\rm{dim}\rm{Wh}_{\co_{\ell(\pi)}}(\pi)=1.
\]
(ii)
The nilpotent orbit $\co_{\ell(\pi)}$ is the wavefront set of $\pi$.
\end{theorem}
\begin{proof}
We prove the theorem by induction on $n$.

(i) By Corollary \ref{d1}, we know that
\begin{itemize}

\item
\[
\rm{Wh}_{\co_{(\ell_1,1,\cdots,1)}}(\pi)\cong \CD_{\ell_1}(\pi)=\pi',
\]
where $\pi'$ is an irreducible representation of $\UU_{n-\ell_1}\fq$ defined in Corollary \ref{d1};
\item
the descent index sequence $\ell(\pi')$ of $\pi'$ is equal to $(\ell_2,\ell_3,\cdots)$.

\end{itemize}
By our induction assumption on $n$, we have
\[
\rm{dim}\rm{Wh}_{\co_{\ell(\pi')}}(\pi')=1.
\]
According to Lemma \ref{ex3}, we have
\[
\rm{dim}\rm{Wh}_{\co_{\ell(\pi)}}(\pi)=\rm{dim}\rm{Wh}_{\co_{(\ell_2,\ell_3,\cdots)}}(\rm{Wh}_{\co_{(\ell_1,1 ,\cdots, 1)}}(\pi))=\rm{dim}\rm{Wh}_{\co_{\ell(\pi')}}(\pi')=1.
\]

(ii) Assume that $\co_\lambda=\co_{(\lambda_1,\lambda_2,\cdots)}$ be the wavefront set of $\pi$. In \cite[Section 1.5]{T}, there is a
geometric refinement of the condition (WF2) in the definition of wavefront sets as follows:
 \begin{itemize}
\item[] (WF$2'$) $\langle \pi,\Gamma_{\gamma'}\rangle\ne 0$ for some $\gamma'=\{X',H',Y'\}$ with $X'\in \wco_{\lambda'}^{F}$ implies $\widetilde{\co}_{\lambda'}\subset\widetilde{\co}_\lambda$.
\end{itemize}
In the unitary groups case, $\widetilde{\co}_{\lambda'}\subset\widetilde{\co}_\lambda$ is equivalent to $\lambda'\le\lambda$, where  $\lambda'\le \lambda $ is the parabolic ordering, for each $k$,
 \[
\sum_{i=1}^k \lambda'_i\le \sum_{i=1}^k\lambda_i.
 \]
Recall that we have already proved that $\rm{dim}\rm{Wh}_{\co_{\ell(\pi')}}(\pi')=1$. So $\lambda_1\ge \ell_1$. If $\lambda_1>\ell_1$, then by Lemma \ref{ex3},
\[
\rm{dim}\rm{Wh}_{\co_{(\lambda_2,\lambda_3,\cdots)}}(\rm{Wh}_{\co_{(\lambda_1,1 ,\cdots, 1)}}(\pi))=\rm{dim}\rm{Wh}_{\co_{\lambda}}(\pi)\ne 0,
\]
which implies that $\rm{Wh}_{\co_{(\lambda_1,1 ,\cdots, 1)}}(\pi)\ne 0$. This contradicts the definition of $\ell(\pi)$. So $\lambda_1=\ell_1$ and
 \[
\rm{Wh}_{\co_{(\lambda_1,1 ,\cdots, 1)}}(\pi)=\rm{Wh}_{\co_{(\ell_1,1,\cdots,1)}}(\pi)=\pi'
\]
 is an irreducible representation of $\UU_{n-\ell_1}\fq$. It immediately follows from Lemma \ref{ex3} that $(\lambda_2,\lambda_3,\cdots)$ is the wavefront set of $\pi$. Then by our induction assumption on $n$, we have $(\lambda_2,\lambda_3,\cdots)=(\ell_2,\ell_3,\cdots)$, which completes the proof.

\end{proof}

If the nilpotent orbit $\co$ is not the wavefront of $\pi$,  in principle, we can still calculate the dimension of $\rm{Wh}_{\co}(\pi)$ in a similar way. The following theorem is the opposite of (WF$2'$).
\begin{theorem}
 Let $\pi\in\rm{Irr}(\UU_n\fq)$, $\ell(\pi)=(\ell_1,\ell_2,\cdots)$ be the descent index of $\pi$. If $\lambda\le \ell(\pi) $, then
  \[
 \rm{dim}\rm{Wh}_{\co_{\lambda}}(\pi)>0.
 \]

\end{theorem}
\begin{proof}
Suppose that $\pi\in\cal{E}(G^F,s)$, and the image of $\pi$ under the Lusztig correspondence is
\[
\cal{L}_s(\pi)=\prod_{i=1}^k\pi{[a_i]},
\]
 where $\pi{[a_i]}$ is the unipotent representation corresponding to partition $\lambda[a_i]$ for each $[a_i]$.

 If $k=n$, then $\pi{[a_i]}$ is a character of $\UU_1\fq$. Therefore, for each $i$, $\pi{[a_i]}$ is a generic representation and so is $\pi$. There is nothing to prove.

 Assume that $k<n$. We prove the theorem by induction on $n$. By our assumption that the finite field $\Fq$ is sufficient, we can pick an element $b\in \Fq^\times$ different from $\{a_i\}$. Denote by $\pi':=\CD_{\ell_1}(\pi)\in \rm{Irr}(\UU_{n-\ell_1}\fq,s')$. Then the set of eigenvalues of the semisimple element $s'$ is a subset of $\{a_i\}$. Let
 \[
\cal{L}_{s'}(\pi')=\prod_{i=1}^k\pi'{[a_i]}.
\]
Consider the semisimple element
 \[
 s_b:=\rm{diag}(b,\cdots,b,s')\in \UU_{n-\lambda_1}\fq.
 \]
 Let $\pi_b\in \cal{E}(\UU_{n-\lambda_1}\fq,s_b)$ such that the image of $\pi$ under the Lusztig correspondence is
 \[
\pi[b]\otimes \prod_{i=1}^k\pi'{[a_i]}
 \]
 where $\pi[b]\cong\pi_{(1,1,\cdots,1)}\in \rm{Irr}(\UU_{\ell_1-\lambda_1}\fq)$ is the component corresponding to $[b]$. By Corollary \ref{d1}, we have
 \[
 \ell(\pi_b)=(\ell_2+\ell_1-\lambda_1,\ell_3,\ell_4,\cdots)\ge (\lambda_2,\lambda_3,\cdots).
 \]
By our induction assumption, we have
\[
 \rm{dim}\rm{Wh}_{\co_{(\lambda_2,\lambda_3,\cdots)}}(\pi_b)>0.
\]
According to Corollary \ref{d1}, we know that $\pi_b$ is an irreducible constitute of the generalized Whittaker model  $\rm{dim}\rm{Wh}_{\co_{(\lambda_1,1,1,\cdots,1)}}(\pi)$. Hence
 \[
 \begin{aligned}
 \rm{dim}\rm{Wh}_{\co_{\lambda}}(\pi)
 = \rm{dim}\rm{Wh}_{\co_{(\lambda_2,\lambda_3,\cdots)}}(\rm{Wh}_{\co_{(\lambda_1,1,1,\cdots,1)}}(\pi))
 \ge \rm{dim}\rm{Wh}_{\co_{(\lambda_2,\lambda_3,\cdots)}}(\pi_b)
 >\ 0.
 \end{aligned}
 \]
 \end{proof}

 \section{Calculation of Generalized Gelfand-Graev Representations}\label{sec8}

 Rainbolt studies the GGGRs of $\UU_3\fq$ in \cite{Ra}.  In the general case, little is known. In this section, we induce an algorithm to calculate the generalized Gelfand-Graev representations of unitary groups. This algorithm is based on our calculation in section \ref{sec5.2}. Let $\co_\lambda=\co_{(\lambda_1,\cdots,\lambda_k)}$ be a nilpotent orbit in $\UU_n\fq$, and $\Gamma_{\lambda}$ be the generalized Gelfand-Graev representation corresponding to the orbit $\co_{\lambda}$. To calculate the generalized Gelfand-Graev representation $\Gamma_{\lambda}$, we only need to calculate
 \[
\rm{dim}\rm{Hom}_{\UU_n\fq}(\pi, \Gamma_\lambda)=\rm{dim}\rm{Wh}_{\co_{\lambda}}(\pi).
 \]
 By Lemma \ref{ex3}, we know that
 \[
 \begin{aligned}
\ & \rm{dim}\rm{Wh}_{\co_{\lambda}}(\pi)\\
=\ &\rm{dim}\rm{Wh}_{\co_{(\lambda_2,\lambda_3,\cdots})}(\rm{Wh}_{\co_{(\lambda_1,1,\cdots,1)}}(\pi))\\
 =\ &\rm{dim}\rm{Wh}_{\co_{(\lambda_3,\lambda_4,\cdots)}}(\rm{Wh}_{\co_{(\lambda_2,1,\cdots,1)}}(\rm{Wh}_{(\co_{\lambda_1,1,\cdots,1)}}(\pi)))\\
  =\ &\rm{dim} \rm{Wh}_{\co_{(\lambda_k,1,\cdots,1)}}(\rm{Wh}_{\co_{(\lambda_{k-1},1,\cdots,1)}}(\rm{Wh}_{\co_{(\lambda_{k-2},1,\cdots,1)}}(\cdots(\rm{Wh}_{\co_{(\lambda_1,1,\cdots,1)}}(\pi)))\\
 \end{aligned}
 \]
 So we can reduce the calculation to this kind of generalized Whittaker model
 \[
 \rm{dim}\rm{Wh}_{\co_{\ell,1,\cdots,1}}(\pi).
 \]
  Let $\pi'\in \rm{Irr}(\UU_{n-\ell}\fq)$. In Section \ref{sec5}, we give a explicit description of
\[
m(\pi,\pi'):=\dim \rm{Hom}_{\rm{U}_{n-\ell}(\Fq)}(\rm{Wh}_{\co_{\ell,1,\cdots,1}}(\pi),\pi').
\]
So, in principle, one can calculate the dimension
of arbitrary generalized Whittaker models. However, in general, generalized Whittaker models contain
a large number of irreducible constituents, and the calculation would be very complicated. For example, we will use this approach to calculate the $\UU_4\fq$ case.

\subsection{Representations of $\UU_4\fq$}

We divide ${\rm {Irr}}(\UU_4\fq)$ into the following 11 cases.
  \begin{itemize}
\item Case ($A$): There are $[a_1]\ne [a_2]\ne [a_3]\ne [a_4]$ such that $\#[a_i]=1$ and $|\cal{L}_\pi([a_i])|=1$;
\item Case ($B$): There are $[a_1]\ne [a_2]\ne [a_3]$ such that $\#[a_i]=1$, $|\cal{L}_\pi([a_1])|=|\cal{L}_\pi([a_2])|=1$, and $|\cal{L}_\pi([a_3])|=2$;
\item Case ($C$): There are $[a_1]\ne [a_2]\ne [a_3]$ such that $\#[a_1]=2$, $\#[a_2]=\#[a_3]=1$, and $|\cal{L}_\pi([a_i])|=1$;
\item Case ($D$): There are $[a_1]\ne [a_2]$ such that $\#[a_i]=1$, and $|\cal{L}_\pi([a_1])|=1$, and $|\cal{L}_\pi([a_2])|=3$;
\item Case ($E$): There are $[a_1]\ne [a_2]$ such that $\#[a_i]=1$, and $|\cal{L}_\pi([a_i])|=2$;
\item Case ($F$): There are $[a_1]\ne [a_2]$ such that $\#[a_1]=2$, $\#[a_2]=1$, $|\cal{L}_\pi([a_1])|=1$, and $|\cal{L}_\pi([a_2])|=2$;
\item Case ($G$): There are $[a_1]\ne [a_2]$ such that $\#[a_1]=\#[a_2]=2$, and $|\cal{L}_\pi([a_i])|=1$;
\item Case ($H$): There are $[a_1]\ne [a_2]$ such that $\#[a_1]=3$, $\#[a_2]=1$, $|\cal{L}_\pi([a_i])|=1$;
\item Case ($I$): There are $[a_1]$ such that $\#[a_1]=1$ and $|\cal{L}_\pi([a_i])|=4$;
\item Case ($J$): There are $[a_1]$ such that $\#[a_1]=2$ and $|\cal{L}_\pi([a_i])|=2$;
\item Case ($K$): There are $[a_1]$ such that $\#[a_1]=4$ and $|\cal{L}_\pi([a_i])|=1$.
\end{itemize}
Let
\[
R(X) :=\{\pi \textrm{ is an irreducible representation in the Case (X)} \}
\]
and
\[
R(X_{[a_1]\to \lambda,[a_2]\to \lambda',\cdots}):=\{\pi\in R(X)\ | \ \cal{L}_\pi([a_1])=\lambda, \cal{L}_\pi([a_2])=\lambda',\cdots\}.
\]
Write
\[
\Pi(X):=\bigoplus_{\pi\in R(X)}\pi
\]
and
\[
\Pi(X_{[a_1]\to \lambda,[a_2]\to \lambda',\cdots}):=\bigoplus_{\pi\in R((X_{[a_1]\to \lambda,[a_2]\to \lambda',\cdots})}\pi.
\]
Let $\cal{X}$ be the collection of symbols $X_{[a_1]\to \lambda,[a_2]\to \lambda',\cdots}$ such that $R(X_{[a_1]\to \lambda,[a_2]\to \lambda',\cdots})\ne \emptyset$.

There are 5 nilpotent orbits in $\UU_4\fq$: $\co_{(4)}$, $\co_{(3,1)}$, $\co_{(2,2)}$, $\co_{(2,1,1)}$, and $\co_{(1,1,1,1)}$. The generalized Gelfand-Graev representation $\Gamma_{(4)}$ is the usual one. And the generalized Gelfand-Graev representation $\Gamma_{(1,1,1,1)}$ is the regular representation of $\UU_4\fq$.

\subsection{Calculation of $\Gamma_{(3,1)}$}
Consider $\Gamma_{3,1}$. Let $\pi'$ be the irreducible representation of $\UU_1\fq$ such that $\Lambda_{\pi'}([b])=(1)$. The representation $\pi'$ is actually a character of $\UU_1\fq$, and $\rm{dim}\pi'=1$
According to the calculation in Section \ref{sec5}, we have the following equations:
  \begin{itemize}
\item If $\pi\in R(A)$, we have
\[
m(\pi, \pi')=\left\{
\begin{array}{ll}
1, &  \textrm{if $[b]\notin\{[a_1], [a_2], [a_3], [a_4]\}$},\\
0, & \textrm{otherwise;}
\end{array}\right.
\]
\item If $\pi\in R(B)$, we have
\[
m(\pi, \pi')=\left\{
\begin{array}{ll}
1, &  \textrm{if $[b]\notin\{[a_1],[a_2],[a_3]\}$, and and $\cal{L}_\pi([a_3])=(1,1)$},\\
1, &  \textrm{if $[b]=[a_3]$, and $\cal{L}_\pi([a_3])=(2)$},\\
0, & \textrm{otherwise;}
\end{array}\right.
\]
\item
If $\pi\in R(C)$, we have
\[
m(\pi, \pi')=\left\{
\begin{array}{ll}
1, &  \textrm{if $[b]\notin\{[a_2],[a_3]\}$},\\
0, & \textrm{otherwise;}
\end{array}\right.
\]
\item If $\pi\in R(D)$, we have
\[
m(\pi, \pi')=\left\{
\begin{array}{ll}
1, &  \textrm{if $[b]\notin\{[a_1],[a_2]\}$, and $\cal{L}_\pi([a_2])=(1,1,1)$},\\
1, &  \textrm{if $[b]=[a_2]$, and $\cal{L}_\pi([a_2])=(2,1)$},\\
0, & \textrm{otherwise;}
\end{array}\right.
\]
\item If $\pi\in R(E)$, we have
\[
m(\pi, \pi')=\left\{
\begin{array}{ll}
1, &  \textrm{if $[b]\notin\{[a_1],[a_2]\}$, and $\cal{L}_\pi([a_1])=\cal{L}_\pi([a_2])=(1,1)$},\\
1, &  \textrm{if $[b]=[a_i]$, and $\cal{L}_\pi([a_i])=(2)$, and $\cal{L}_\pi([a_{i'}])=(1,1)$},\\
0, & \textrm{otherwise;}
\end{array}\right.
\]
\item If $\pi\in R(F)$, we have
\[
m(\pi, \pi')=\left\{
\begin{array}{ll}
1, &  \textrm{if $[b]\ne[a_2]$, and $\cal{L}_\pi([a_2])=(1,1)$},\\
1, &  \textrm{if $[b]=[a_2]$, and $\cal{L}_\pi([a_2])=(2)$},\\
0, & \textrm{otherwise;}
\end{array}\right.
\]
\item If $\pi\in R(G)$, then $m(\pi, \pi')\equiv1$;
\item If $\pi\in R(H)$, we have
\[
m(\pi, \pi')=\left\{
\begin{array}{ll}
1, &  \textrm{if $[b]\ne[a_2]$},\\
0, & \textrm{otherwise;}
\end{array}\right.
\]
\item If $\pi\in R(I)$, we have
\[
m(\pi, \pi')=\left\{
\begin{array}{ll}
1, &  \textrm{if $[b]\ne[a_1]$, and $\cal{L}_\pi[a_1]=(1,1,1,1)$},\\
1, &  \textrm{if $[b]=[a_1]$, and $\cal{L}_\pi[a_1]=(2,1,1)$},\\
0, & \textrm{otherwise;}
\end{array}\right.
\]
\item If $\pi\in R(J)$, \[
m(\pi, \pi')=\left\{
\begin{array}{ll}
1, &  \textrm{if $\cal{L}_\pi[a_1]=(1,1)$},\\
0, & \textrm{otherwise;}
\end{array}\right.
\]
\item If $\pi\in R(K)$, then $m(\pi, \pi')\equiv1$.
\end{itemize}

Let
\[
f_{(3,1)}(X_{[a_1]\to \lambda,[a_2]\to \lambda',\cdots}):=\bigoplus_{\pi'\in\rm{Irr}(\UU_1\fq)}\rm{dim}\pi'\cdot m(\pi,\pi')
\]
where $\pi\in R(X_{[a_1]\to \lambda,[a_2]\to \lambda',\cdots})$. By the above calculation, the function $f$ is well-defined. We have
\[
\begin{matrix}
f_{(3,1)}(A)&= &q-5;\\
f_{(3,1)}(B_{[a_3]\to(1,1)]})&=&q-4;\\
f_{(3,1)}(B_{[a_3]\to(2)]})&=&1;&\\
f_{(3,1)}(C)&=&q-3;\\
f_{(3,1)}(D_{[a_2]\to (1,1,1)})&=&q-3;\\
f_{(3,1)}(D_{[a_2]\to (2,1)})&=&1;\\
f_{(3,1)}(D_{[a_2]\to (3)})&=&0;\\
f_{(3,1)}(E_{[a_1]\to(1,1),[a_2]\to(1,1)})&=&q-3;\\
f_{(3,1)}(E_{[a_1]\to(2),[a_2]\to(1,1)})&=&1;\\
f_{(3,1)}(E_{[a_1]\to(2),[a_2]\to(2)})&=&0;\\
f_{(3,1)}(F_{[a_2]\to(1,1)})&=&q-2;\\
f_{(3,1)}(F_{[a_2]\to(2)})&=&1;\\
f_{(3,1)}(G)&=&(q-1);\\
f_{(3,1)}(H)&=&q-2;\\
f_{(3,1)}(I_{[a_1]\to(1,1,1,1)})&=&q-2&;\\
f_{(3,1)}(I_{[a_1]\to(2,1,1)})&=&1;\\
f_{(3,1)}(I_{[a_1]\to(2,2)})&=&0;\\
f_{(3,1)}(I_{[a_1]\to(3,1)})&=&0;\\
f_{(3,1)}(I_{[a_1]\to(4)})&=&0;\\
f_{(3,1)}(J_{[a_1]\to(1,1)})&=&q-1;\\
f_{(3,1)}(J_{[a_1]\to(2)})&=&0;\\
f_{(3,1)}(K)&=&q-1,\\
\end{matrix}
\]
and
\[
\begin{aligned}
\Gamma_{(3,1)}=\bigoplus_{X_{[a_1]\to \lambda,[a_2]\to \lambda',\cdots}\in\cal{X}}f_{(3,1)}(X_{[a_1]\to \lambda,[a_2]\to \lambda',\cdots})\Pi(X_{[a_1]\to \lambda,[a_2]\to \lambda',\cdots}).
\end{aligned}
\]

\subsection{Calculation of $\Gamma_{(2,1,1)}$}
The generalized Whittaker model $\rm{Wh}_{(2,1,1)}(\pi)$ can be regard as a representation of $\UU_2\fq$.
We classify representations of $\UU_2$ as follows. Let $\pi'\in \rm{Irr}(\UU_2\fq)$.
  \begin{itemize}
\item Case ($A'$): There is $[b]$ such that $\#[b]=1$, and $\Lambda_{\pi'}([b])=(2)$;
\item Case ($B'$): There is $[b]$ such that $\#[b]=1$, and $\Lambda_{\pi'}([b])=(1,1)$;
\item Case ($C'$): There is $[b_1]$ and $[b_2]$ such that $\#[b_1]=[b_2]=1$;
\item Case ($D'$): There is $[b]$ such that $\#[b]=2$.
\end{itemize}
We define $R(X')$ and $\Pi(X')$ similarly. Then $|R(A')|=|R(B')|=q-1$, $|R(C')|=(q-1)(q-2)$, and $|R(D)|=q^2-1$.

Let $\pi'\in R(A')$. Then $\rm{dim}\pi'=1$.
According to Theorem \ref{ggp}, we have the following equations:
  \begin{itemize}
\item If $\pi\in R(A)$, we have
\[
m(\pi, \pi')=\left\{
\begin{array}{ll}
1, &  \textrm{if $[b]=[a_i]$},\\
0, & \textrm{otherwise;}
\end{array}\right.
\]
\item If $\pi\in R(B)$, we have
\[
m(\pi, \pi')=\left\{
\begin{array}{ll}
1, &  \textrm{if $[b]\in\{[a_1],[a_2]\}$ },\\
1, &  \textrm{if $[b]=[a_3]$ and $\cal{L}_\pi[a_3]=(1,1)$},\\
0, & \textrm{otherwise;}
\end{array}\right.
\]
\item
If $\pi\in R(C)$, we have
\[
m(\pi, \pi')=\left\{
\begin{array}{ll}
1, &  \textrm{if $[b]\in\{[a_2],[a_3]\}$},\\
0, & \textrm{otherwise;}
\end{array}\right.
\]
\item If $\pi\in R(D)$, we have
\[
m(\pi, \pi')=\left\{
\begin{array}{ll}
1, &  \textrm{if $[b]=[a_1]$, and $\cal{L}_\pi[a_2]\in\{(1,1,1)$}\\
1, &  \textrm{if $[b]=[a_3]$, and $\cal{L}_\pi[a_2]\in\{(1,1,1),(3)\}$},\\
0, & \textrm{otherwise;}
\end{array}\right.
\]
\item If $\pi\in R(E)$, we have
\[
m(\pi, \pi')=\left\{
\begin{array}{ll}
1, &  \textrm{if $[b]\in\{[a_1],[a_2]\}$, and $\cal{L}_\pi([a_1])=\cal{L}_\pi([a_2])=(1,1)$},\\
0, & \textrm{otherwise;}
\end{array}\right.
\]
\item If $\pi\in R(F)$, we have
\[
m(\pi, \pi')=\left\{
\begin{array}{ll}
1, &  \textrm{if $[b]=[a_2]$, and $\cal{L}_\pi([a_1])=(1,1)$},\\
0, & \textrm{otherwise;}
\end{array}\right.
\]
\item If $\pi\in R(G)$, then $m(\pi, \pi')\equiv0$;
\item If $\pi\in R(H)$, we have
\[
m(\pi, \pi')=\left\{
\begin{array}{ll}
1, &  \textrm{if $[b]=[a_2]$},\\
0, & \textrm{otherwise;}
\end{array}\right.
\]
\item If $\pi\in R(I)$, we have
\[
m(\pi, \pi')=\left\{
\begin{array}{ll}
1, &  \textrm{if $[b]= [a_1]$, and $\cal{L}_\pi[a_1]\in\{(3,1),(1,1,1,1)\}$},\\
0, & \textrm{otherwise;}
\end{array}\right.
\]
\item If $\pi\in R(J)$, then $m(\pi, \pi')\equiv0$;
\item If $\pi\in R(K)$, then $m(\pi, \pi')\equiv0$.
\end{itemize}

Let $\pi'\in R(B')$. Then $\rm{dim}\pi'=q$.
According to Theorem \ref{ggp}, we have the following equations:
  \begin{itemize}
\item If $\pi\in R(A)$, we have
\[
m(\pi, \pi')=\left\{
\begin{array}{ll}
1, &  \textrm{if $[b]\notin\{[a_1],[a_2],[a_3],[a_4]\}$},\\
0, & \textrm{otherwise;}
\end{array}\right.
\]
\item If $\pi\in R(B)$, we have
\[
m(\pi, \pi')=\left\{
\begin{array}{ll}
1, &  \textrm{if $[b]\notin\{[a_1],[a_2],[a_3]\}$, and and $\cal{L}_\pi([a_3])=(1,1)$},\\
1, &  \textrm{if $[b]=[a_3]$},\\
0, & \textrm{otherwise;}
\end{array}\right.
\]
\item
If $\pi\in R(C)$, we have
\[
m(\pi, \pi')=\left\{
\begin{array}{ll}
1, &  \textrm{if $[b]\notin\{[a_2],[a_3]\}$},\\
0, & \textrm{otherwise;}
\end{array}\right.
\]
\item If $\pi\in R(D)$, we have
\[
m(\pi, \pi')=\left\{
\begin{array}{ll}
1, &  \textrm{if $[b]\ne[a_1]$, and $\cal{L}_\pi([a_2])=(1,1,1)$},\\
0, & \textrm{otherwise;}
\end{array}\right.
\]
\item If $\pi\in R(E)$, we have
\[
m(\pi, \pi')=\left\{
\begin{array}{ll}
1, &  \textrm{if $[b]\notin\{[a_1],[a_2]\}$, and $\cal{L}_\pi([a_1])=\cal{L}_\pi([a_2])=(1,1)$},\\
1, &  \textrm{if $[b]=[a_i]$, and $\cal{L}_\pi([a_i'])=(1,1)$},\\
0, & \textrm{otherwise;}
\end{array}\right.
\]
\item If $\pi\in R(F)$, we have
\[
m(\pi, \pi')=\left\{
\begin{array}{ll}
1, &  \textrm{if $[b]\ne[a_2]$, and $\cal{L}_\pi([a_2])=(1,1)$},\\
1, &  \textrm{if $[b]=[a_2]$},\\
0, & \textrm{otherwise;}
\end{array}\right.
\]
\item If $\pi\in R(G)$, then $m(\pi, \pi')\equiv1$;
\item If $\pi\in R(H)$, we have
\[
m(\pi, \pi')=\left\{
\begin{array}{ll}
1, &  \textrm{if $[b]\ne [a_2]$},\\
0, & \textrm{otherwise;}
\end{array}\right.
\]
\item If $\pi\in R(I)$, we have
\[
m(\pi, \pi')=\left\{
\begin{array}{ll}
1, &  \textrm{if $\cal{L}_\pi[a_1]=(1,1,1,1)$},\\
1, &  \textrm{if $[b]= [a_1]$, and $\cal{L}_\pi[a_1]=(2,2)$},\\
0, & \textrm{otherwise;}
\end{array}\right.
\]
\item If $\pi\in R(J)$, we have
\[
m(\pi, \pi')=\left\{
\begin{array}{ll}
1, &  \textrm{if $\cal{L}_\pi[a_1]=(1,1)$},\\
0, & \textrm{otherwise;}
\end{array}\right.
\]
\item If $\pi\in R(K)$, then $m(\pi, \pi')\equiv1$.
\end{itemize}

Let $\pi'\in R(C')$. Then $\rm{dim}\pi'=q-1$.
According to Theorem \ref{ggp}, we have the following equations:
  \begin{itemize}
\item If $\pi\in R(A)$, we have
\[
m(\pi, \pi')=\left\{
\begin{array}{ll}
1, &  \textrm{if $[b_1],[b_2]\notin\{[a_1],[a_2],[a_3],[a_4]\}$},\\
0, & \textrm{otherwise;}
\end{array}\right.
\]
\item If $\pi\in R(B)$, we have
\[
m(\pi, \pi')=\left\{
\begin{array}{ll}
1, &  \textrm{if $[b_1],[b_2]\notin\{[a_1],[a_2],[a_3]\}$,  and $\cal{L}_\pi([a_3])=(1,1)$},\\
1, &  \textrm{if $[b_1]\notin\{[a_1],[a_2],[a_3]\}$, and $[b_2]=[a_3]$, and $\cal{L}_\pi([a_3])=(2)$},\\
0, & \textrm{otherwise;}
\end{array}\right.
\]
\item
If $\pi\in R(C)$, we have
\[
m(\pi, \pi')=\left\{
\begin{array}{ll}
1, &  \textrm{if $[b_1],[b_2]\notin\{[a_2],[a_3]\}$},\\
0, & \textrm{otherwise;}
\end{array}\right.
\]
\item If $\pi\in R(D)$, we have
\[
m(\pi, \pi')=\left\{
\begin{array}{ll}
1, &  \textrm{if $[b_1],[b_2]\notin\{[a_1],[a_2]\}$, and $\cal{L}_\pi([a_2])=(1,1,1)$},\\
1, &  \textrm{if $[b_1]\notin\{[a_1],[a_2]\}$, and $[b_2]=[a_2]$, and $\cal{L}_\pi([a_2])=(2,1)$},\\
0, & \textrm{otherwise;}
\end{array}\right.
\]
\item If $\pi\in R(E)$, we have
\[
m(\pi, \pi')=\left\{
\begin{array}{ll}
1, &  \textrm{if $[b_1],[b_2]\notin\{[a_1],[a_2]\}$, and $\cal{L}_\pi([a_1])=\cal{L}_\pi([a_2])=(1,1)$},\\
1, &  \textrm{if $[b_1]\notin\{[a_1],[a_2]\}$, and $[b_2]=[a_i]$, and $\cal{L}_\pi([a_i])=(2)$, and $\cal{L}_\pi([a_i'])=(1,1)$},\\
1, &  \textrm{if $[b_1]=[a_1]$, $[b_2]=[a_2]$ and $\cal{L}_\pi([a_1])=\cal{L}_\pi([a_2]) =(2)$},\\
0, & \textrm{otherwise;}
\end{array}\right.
\]
\item If $\pi\in R(F)$, we have
\[
m(\pi, \pi')=\left\{
\begin{array}{ll}
1, &  \textrm{if $[b_1]\ne [a_2]$, and $[b_2]\ne[a_2]$, and $\cal{L}_\pi([a_2])=(1,1)$},\\
1, &  \textrm{if $[b_1]\ne[a_2]$, and $[b_2]=[a_2]$, and $\cal{L}_\pi([a_2])=(2)$},\\
0, & \textrm{otherwise;}
\end{array}\right.
\]
\item If $\pi\in R(G)$, then $m(\pi, \pi')\equiv1$;
\item If $\pi\in R(H)$, we have
\[
m(\pi, \pi')=\left\{
\begin{array}{ll}
1, &  \textrm{if $[b_1]\ne[a_2]$, and $[b_2]\ne [a_2]$},\\
0, & \textrm{otherwise;}
\end{array}\right.
\]
\item If $\pi\in R(I)$, we have
\[
m(\pi, \pi')=\left\{
\begin{array}{ll}
1, &  \textrm{if $[b_1]\ne[a_1]$, $[b_2]\ne[a_1]$, and $\cal{L}_\pi[a_1]=(1,1,1,1)$},\\
1, &  \textrm{if $[b_1]\ne [a_1]$, $[b_2]=[a_1]$, and $\cal{L}_\pi[a_1]=(2,1,1)$},\\
0, & \textrm{otherwise;}
\end{array}\right.
\]
\item If $\pi\in R(J)$, we have
\[
m(\pi, \pi')=\left\{
\begin{array}{ll}
1, &  \textrm{if $\cal{L}_\pi[a_1]=(1,1)$},\\
0, & \textrm{otherwise;}
\end{array}\right.
\]
\item If $\pi\in R(K)$, then $m(\pi, \pi')\equiv1$.
\end{itemize}

Let $\pi'\in R(D')$. Then $\rm{dim}\pi'=q+1$.
According to Theorem \ref{ggp}, we have the following equations:
  \begin{itemize}
\item If $\pi\in R(A)$, then
$m(\pi, \pi')\equiv1$;
\item If $\pi\in R(B)$, we have
\[
m(\pi, \pi')=\left\{
\begin{array}{ll}
1, &  \textrm{if $\cal{L}_\pi([a_3])=(1,1)$}\\
0, & \textrm{otherwise;}
\end{array}\right.
\]
\item
If $\pi\in R(C)$, we have
\[
m(\pi, \pi')=\left\{
\begin{array}{ll}
1, &  \textrm{if $[b]\ne [a_1]$},\\
0, & \textrm{otherwise;}
\end{array}\right.
\]
\item If $\pi\in R(D)$, we have
\[
m(\pi, \pi')=\left\{
\begin{array}{ll}
1, &  \textrm{if $\cal{L}_\pi([a_2])=(1,1,1)$},\\
0, & \textrm{otherwise;}
\end{array}\right.
\]
\item If $\pi\in R(E)$, we have
\[
m(\pi, \pi')=\left\{
\begin{array}{ll}
1, &  \textrm{if $\cal{L}_\pi([a_1])=\cal{L}_\pi([a_2])=(1,1)$},\\
0, & \textrm{otherwise;}
\end{array}\right.
\]
\item If $\pi\in R(F)$, we have
\[
m(\pi, \pi')=\left\{
\begin{array}{ll}
1, &  \textrm{if $[b]\ne [a_1]$, and $\cal{L}_\pi([a_2])=(1,1)$},\\
0, & \textrm{otherwise;}
\end{array}\right.
\]
\item If $\pi\in R(G)$,
we have
\[
m(\pi, \pi')=\left\{
\begin{array}{ll}
1, &  \textrm{if $[b]\notin\{[a_1],[a_2]\}$},\\
0, & \textrm{otherwise;}
\end{array}\right.
\]
\item If $\pi\in R(H)$, then $m(\pi, \pi')\equiv1$;
\item If $\pi\in R(I)$, we have
\[
m(\pi, \pi')=\left\{
\begin{array}{ll}
1, &  \textrm{if $\cal{L}_\pi[a_1]=(1,1,1,1)$},\\
0, & \textrm{otherwise;}
\end{array}\right.
\]
\item If $\pi\in R(J)$, we have
\[
m(\pi, \pi')=\left\{
\begin{array}{ll}
1, &  \textrm{if $[b]\ne[a_1]$ and $\cal{L}_\pi([a_1])=(1,1)$},\\
1, &  \textrm{if $[b]=[a_1]$ and $\cal{L}_\pi([a_1])=(2)$},\\
0, & \textrm{otherwise;}
\end{array}\right.
\]
\item If $\pi\in R(K)$, then $m(\pi, \pi')\equiv1$.
\end{itemize}

Let
\[
f_{(2,1,1)}(X_{[a_1]\to \lambda,[a_2]\to \lambda',\cdots}):=\bigoplus_{\pi'\in\rm{Irr}(\UU_2\fq)}\rm{dim}\pi'\cdot m(\pi,\pi')
\]
where $\pi\in R(X_{[a_1]\to \lambda,[a_2]\to \lambda',\cdots})$. By the above calculation, the function $f$ is well-defined. We have
\[
\begin{matrix}
f_{(2,1,1)}(A)&= &4&+&q(q-5)&+&(q-1)(q-5)(q-6)&+&(q+1)(q^2-1);\\
f_{(2,1,1)}(B_{[a_3]\to(1,1)]})&=&3&+&q(q-3)&+&(q-1)(q-4)(q-5)&+&(q+1)(q^2-1);\\
f_{(2,1,1)}(B_{[a_3]\to(2)]})&=&2&+&q&+&(q-1)(q-4);&\\
f_{(2,1,1)}(C)&=&2&+&q(q-3)&+&(q-1)(q-3)(q-4)&+&(q+1)(q^2-2);\\
f_{(2,1,1)}(D_{[a_2]\to (1,1,1)})&=&2&+&q(q-1)&+&(q-1)(q-3)(q-4)&+&(q+1)(q^2-1);\\
f_{(2,1,1)}(D_{[a_2]\to (2,1)})&=&&&&&(q-1)(q-3);\\
f_{(2,1,1)}(D_{[a_2]\to (3)})&=&1;\\
f_{(2,1,1)}(E_{[a_1]\to(1,1),[a_2]\to(1,1)})&=&2&+&q(q-1)&+&(q-1)(q-3)(q-4)&+&(q+1)(q^2-1);\\
f_{(2,1,1)}(E_{[a_1]\to(2),[a_2]\to(1,1)})&=&&&q&+&2(q-1)(q-3);\\
f_{(2,1,1)}(E_{[a_1]\to(2),[a_2]\to(2)})&=&&&&&q-1;\\
f_{(2,1,1)}(F_{[a_2]\to(1,1)})&=&1&+&q(q-1)&+&(q-1)(q-2)(q-3)&+&(q+1)(q^2-2);\\
f_{(2,1,1)}(F_{[a_2]\to(2)})&=&&&q&+&(q-1)(q-2);\\
f_{(2,1,1)}(G)&=&&&q(q-1)&+&(q-1)^2(q-2)&+&(q+1)(q^2-3)&;\\
f_{(2,1,1)}(H)&=&1&+&q(q-2)&+&(q-1)(q-2)(q-3)&+&(q+1)(q^2-1);\\
f_{(2,1,1)}(I_{[a_1]\to(1,1,1,1)})&=&1&+&q(q-1)&+&(q-1)(q-2)(q-3)&+&(q+1)(q^2-1)&;\\
f_{(2,1,1)}(I_{[a_1]\to(2,1,1)})&=&&&&&(q-1)(q-2);\\
f_{(2,1,1)}(I_{[a_1]\to(2,2)})&=&&&q;\\
f_{(2,1,1)}(I_{[a_1]\to(3,1)})&=&1;\\
f_{(2,1,1)}(I_{[a_1]\to(4)})&=&0;\\
f_{(2,1,1)}(J_{[a_1]\to(1,1)})&=&&&q(q-1)&+&(q-1)^2(q-2)&+&(q+1)(q^2-2);\\
f_{(2,1,1)}(J_{[a_1]\to(2)})&=&&&&&&&q+1;\\
f_{(2,1,1)}(K)&=&&&q(q-1)&+&(q-1)^2(q-2)&+&(q+1)(q^2-1),\\
\end{matrix}
\]
and
\[
\begin{aligned}
\Gamma_{(2,1,1)}=\bigoplus_{X_{[a_1]\to \lambda,[a_2]\to \lambda',\cdots}\in\cal{X}}f_{(2,1,1)}(X_{[a_1]\to \lambda,[a_2]\to \lambda',\cdots})\Pi(X_{[a_1]\to \lambda,[a_2]\to \lambda',\cdots}).
\end{aligned}
\]

\subsection{Calculation of $\Gamma_{(2,2)}$}
Recall that for $\pi\in \rm{Irr}(\UU_4\fq)$, we have
\[
\rm{dim}\rm{Wh}_{\co_{(2,2)}}(\pi)=\rm{dim}\rm{Wh}_{\co_{(2)}}(\rm{Wh}_{\co_{(2,1,1)}}(\pi)).
\]
Note that for $\pi'\in \rm{Irr}(\UU_2\fq)$, we have
\[
\rm{dim}\rm{Wh}_{\co_{(2)}}(\pi')=\left\{
\begin{array}{ll}
1, &  \textrm{if $\pi'$ is generic},\\
0, & \textrm{otherwise;}
\end{array}\right.
\]
The description of $\Gamma_{(2,2)}$ immediately follows from the description of $\Gamma_{(2,1,1)}$. We define the function $f_{(2,2)}$ in a similar way. We get $f_{(2,2)}(X_{[a_1]\to \lambda,[a_2]\to \lambda',\cdots})$ from $f_{(2,1,1)}(X_{[a_1]\to \lambda,[a_2]\to \lambda',\cdots})$ by removing its first term and the first factor of the remaining terms. Indeed, the first term corresponds to non-generic representations, and the first factors remaining terms are the dimension of the corresponding generic representations. Then we have
\[
\begin{matrix}
f_{(2,2)}(A)&= &&&q-5&+&(q-5)(q-6)&+&q^2-1;\\
f_{(2,2)}(B_{[a_3]\to(1,1)]})&=&&&q-3&+&(q-4)(q-5)&+&q^2-1;\\
f_{(2,2)}(B_{[a_3]\to(2)]})&=&&&1&+&q-4;&\\
f_{(2,2)}(C)&=&&&q-3&+&(q-3)(q-4)&+&q^2-2;\\
f_{(2,2)}(D_{[a_2]\to (1,1,1)})&=&&&q-1&+&(q-3)(q-4)&+&q^2-1;\\
f_{(2,2)}(D_{[a_2]\to (2,1)})&=&&&&&q-3;\\
f_{(2,2)}(D_{[a_2]\to (3)})&=&&&0;\\
f_{(2,2)}(E_{[a_1]\to(1,1),[a_2]\to(1,1)})&=&&&q-1&+&(q-3)(q-4)&+&q^2-1;\\
f_{(2,2)}(E_{[a_1]\to(2),[a_2]\to(1,1)})&=&&&1&+&2(q-3);\\
f_{(2,2)}(E_{[a_1]\to(2),[a_2]\to(2)})&=&&&&&1;\\
f_{(2,2)}(F_{[a_2]\to(1,1)})&=&&&q-1&+&(q-2)(q-3)&+&q^2-2;\\
f_{(2,2)}(F_{[a_2]\to(2)})&=&&&1&+&q-2;\\
f_{(2,2)}(G)&=&&&q-1&+&(q-1)(q-2)&+&q^2-3;\\
f_{(2,2)}(H)&=&&&q-2&+&(q-2)(q-3)&+&q^2-1;\\
f_{(2,2)}(I_{[a_1]\to(1,1,1,1)})&=&&&q-1&+&(q-2)(q-3)&+&q^2-1;\\
f_{(2,2)}(I_{[a_1]\to(2,1,1)})&=&&&&&q-2;\\
f_{(2,2)}(I_{[a_1]\to(2,2)})&=&&&1;\\
f_{(2,2)}(I_{[a_1]\to(3,1)})&=&&&0;\\
f_{(2,2)}(I_{[a_1]\to(4)})&=&&&0;\\
f_{(2,2)}(J_{[a_1]\to(1,1)})&=&&&q-1&+&(q-1)(q-2)&+&q^2-2;\\
f_{(2,2)}(J_{[a_1]\to(2)})&=&&&0;\\
f_{(2,2)}(K)&=&&&q-1&+&(q-1)(q-2)&+&q^2-1,\\
\end{matrix}
\]
and
\[
\begin{aligned}
\Gamma_{(2,2))}=\bigoplus_{X_{[a_1]\to \lambda,[a_2]\to \lambda',\cdots}\in\cal{X}}f_{(2,2)}(X_{[a_1]\to \lambda,[a_2]\to \lambda',\cdots})\Pi(X_{[a_1]\to \lambda,[a_2]\to \lambda',\cdots}).
\end{aligned}
\]

\end{document}